\documentclass[reqno,12pt]{amsart}

\usepackage[utf8]{inputenc}

\usepackage{enumerate}
\usepackage[margin=1in,marginparwidth=0.7in]{geometry}
\usepackage{ifpdf}
\usepackage{amsmath}
\usepackage{amsfonts}
\usepackage{amssymb}
\usepackage{amsthm}
\usepackage[ocgcolorlinks,hyperfootnotes=false,colorlinks=true,citecolor=blue,linkcolor=blue,urlcolor=blue]{hyperref}
\usepackage{setspace}
\usepackage{amsrefs}
\usepackage{nicefrac}
\usepackage{graphicx}
\usepackage{color}
\usepackage{mathtools}
\newcommand{\ignore}[1]{}

\renewcommand{\Im}{\operatorname{Im}}

\newcommand{\abs}[1]{\left\lvert {#1} \right\rvert}
\newcommand{\sabs}[1]{\lvert {#1} \rvert}

\newcommand{\snorm}[1]{\lVert {#1} \rVert}

\newcommand{\C}{{\mathbb{C}}}
\newcommand{\R}{{\mathbb{R}}}

\newcommand{\sV}{{\mathcal{V}}}

\newcommand{\rank}{\operatorname{rank}}

\newtheorem{thm}{Theorem}[section]

\newtheorem{prop}[thm]{Proposition}

\newtheorem{cor}[thm]{Corollary}

\newtheorem{lemma}[thm]{Lemma}

\newtheorem{claim}[thm]{Claim}

\theoremstyle{definition}

\newtheorem{example}[thm]{Example}

\theoremstyle{remark}
\newtheorem{remark}[thm]{Remark}

\newcommand{\avoidbreak}{\postdisplaypenalty=100}

\author{Ji\v{r}\'{\i} Lebl}
\address{Department of Mathematics, Oklahoma State University,
Stillwater, OK 74078, USA}
\email{lebl@math.okstate.edu}

\author{Alan Noell}
\address{Department of Mathematics, Oklahoma State University,
Stillwater, OK 74078, USA}
\email{noell@math.okstate.edu}

\author{Sivaguru Ravisankar}
\address{Tata Institute of Fundamental Research, Centre for Applicable Mathematics, Bengaluru 560065, India}
\email{sivaguru@tifrbng.res.in}

\date{April 14, 2021}

\ifpdf
\hypersetup{
  pdftitle={A CR singular analogue of Severi's theorem},
  pdfauthor={Jiri Lebl, Alan Noell, Sivaguru Ravisankar},
  pdfsubject={Several Complex Variables},
  pdfkeywords={CR singular, holomorphic hull, CR function},
}
\fi

\title{A CR singular analogue of Severi's theorem}


\begin{document}

\begin{abstract}
Real-analytic CR functions on real-analytic CR singular submanifolds 
are not in general restrictions of holomorphic functions, unlike in the CR
nonsingular case.
We give a simple condition
that completely characterizes those quadric CR singular manifolds
of codimension 2 in $\C^{n+1}$ for which an extension result holds.
Consequently, we obtain an extension result for general real-analytic
CR singular submanifolds of codimension 2.
As applications we give a condition for the flattening of such submanifolds,
and we classify CR singular images of CR submanifolds up to second order.
\end{abstract}

\maketitle

%\enlargethispage{\baselineskip}

%%%%%%%%%%%%%%%%%%%%%%%%%%%%%%%%%%%%%%%%%%%%%%%%%%%%%%%%%%%%%%%%%%%%%%%%%%%%

\section{Introduction} \label{section:intro}

Given a submanifold $M$ of $\C^{n+1}$, it is natural to ask  when a function
$f$ defined on $M$ is a restriction to $M$ of a holomorphic function.
A holomorphic function must satisfy the Cauchy--Riemann (CR) equations, and hence
$f$ must satisfy the CR equations in the directions that are tangent to
$M$.  A submanifold is CR if the CR vectors tangent to $M$ give rise to
a vector bundle on $M$; otherwise, $M$ is CR singular.

A classical result of Severi~\cite{Severi:31} is that,
given a real-analytic CR submanifold
$M$ and a real-analytic CR function $f$ defined on $M$, then $f$ is locally the
restriction of a holomorphic function.  For CR singular submanifolds this
theorem does not hold in general.  Every hypersurface is CR, so we
consider the lowest codimension where CR singularities arise, that is,
codimension 2.
A codimension-2 CR singular submanifold $M$ in $\C^{n+1}$ 
with a CR singularity at the origin can
locally, after a rotation by a unitary,
  be written as
\begin{equation} \label{eq:eq1}
w = \rho(z,\bar{z})
\end{equation}
for coordinates $(z,w) \in \C^n \times \C$, where $\rho \in O(\snorm{z}^2)$.
Note that $\rho$ is not necessarily real-valued, and
the complex equation \eqref{eq:eq1} is two real equations.
We will consider a function $f$ to be 
a CR function if it is killed by any CR vector field defined on $M$ (see
\S\ref{sec:prelim}).

A condition on $M$ to guarantee extension for all CR functions
is only possible when $n \geq 2$.
If $n=1$, $M$ has no CR structure outside the origin.
When $\rho$ has
any nonholomorphic quadratic term, we write $M$ as (see~\cite{Bishop65})
\begin{equation}
w = \sabs{z}^2 + \lambda (z^2+\bar{z}^2) + E(z,\bar{z}),
\end{equation}
where $E \in O(\sabs{z}^3)$ and $0 \leq \lambda \leq \infty$
(with $\infty$ interpreted appropriately as $w = z^2+\bar{z}^2 + E$).
Because the manifold has no CR structure outside the origin, the function $\bar{z}$ is vacuously a CR function,
and it never extends to a holomorphic function on a neighborhood of the origin.
It is possible to add extra conditions on both $M$ and $f$
to guarantee an extension,
as we did in \cite{crext1}, but that is not the purpose of the present
paper.

Harris~\cite{Harris} proved 
a necessary and sufficient condition on $f$
for the extension to hold on general $M$,
but the condition is difficult to verify.
In Lebl--Minor--Shroff--Son--Zhang~\cite{LMSSZ} it was proved that if $M$
is an image of a CR submanifold, then no extension result holds.
Given these two results it is perhaps surprising that an extension
result, without extra conditions on $f$, holds generically.

We have previously studied the
extension result when $M$ is flat, a subset of $\C^n \times \R$,
in a series of papers
\cites{crext1,crext2,crext3,crext4} for various $n$ and various regularities.
In this work, we study the real-analytic nonflat case for $n\geq 2$,
dropping the nondegeneracy condition and finding instead  
necessary and sufficient conditions on the quadratic terms of $\rho$
that allow an extension result.

Real codimension-2 CR singular submanifolds were
first studied in $\C^2$ ($n=1$) by
E.~Bishop~\cite{Bishop65}.  The work, focused mostly on the normal form,
was extended by
Moser--Webster~\cite{MoserWebster83},
Moser~\cite{Moser85},
Kenig--Webster
\cites{KenigWebster:82}, Gong~\cite{Gong94:duke},
Huang--Krantz~\cite{HuangKrantz95}, 
Huang--Yin~\cite{HuangYin09}, Slapar~\cite{Slapar:16}, and many others.

For $n\geq 2$ the work is more recent, again focusing mostly on 
the normal form.
See Huang--Yin~\cites{HuangYin09:codim2,HuangYin:flattening1,HuangYin:flattening2},
Gong--Lebl~\cite{GongLebl}, Coffman~\cites{Coffman}, and
Burcea~\cites{Burcea,Burcea2}.
In particular, for $n\geq 2$ it is not possible in general to flatten $M$,
that is, to change
variables to realize $M$ locally
as a submanifold of $\C^n \times \R$.
See Dolbeault--Tomassini--Zaitsev~\cites{DTZ,DTZ2},
Huang--Yin~\cites{HuangYin:flattening1,HuangYin:flattening2}, and
Fang--Huang~\cites{FangHuang}.  The flattening can be obtained
as a holomorphic extension of a function that is the first integral
of the singular foliation by CR orbits (if it exists).

Our main result is the optimal condition on the quadratic part of $\rho$
to guarantee extension.  We state our results for $n \geq 2$,
although they would hold vacuously for $n=1$.  In the theorems, consider
$z$ as a column vector, $z^t$ the transpose, and $z^*$ the conjugate
transpose.

\begin{thm}\label{thm:mainext}
Let $(z,w) \in \C^n \times \C$, $n \geq 2$, be the coordinates and,
near the origin, let $M \subset \C^{n+1}$ be a codimension-2
submanifold given by
\begin{equation} \label{eq:formofM}
w = \rho(z,\bar{z}) = Q(z,\bar{z}) + E(z,\bar{z}) =
z^* A z + \overline{z^t B z} + z^t C z +
E(z, \bar{z}),
\end{equation}
where $\rho$ is real-analytic, % and $O(\snorm{z}^2)$,
$A,B,C$ are complex $n \times n$ matrices,  $B$ and $C$ are symmetric,
and $E$ is
$O(\snorm{z}^3)$.
Assume 
\begin{equation} \label{eq:theABcondition}
   \rank \begin{bmatrix} A^* \\ B \end{bmatrix} \geq 2 .
\end{equation}
   Suppose $f(z,\bar{z})$ is a real-analytic function defined near the origin
   that, when considered as a function on $M$ (parametrized by $z$), is a CR
   function on $M_{\mathit{CR}}$.
   Then there exists a unique holomorphic function $F(z,w)$ defined near the
   origin such that $f$ and $F$ agree on $M$, that~is,
   \begin{equation}
     f(z,\bar{z}) = F\bigl(z, \rho(z,\bar{z}) \bigr) .
   \end{equation}
\end{thm}

Any real-analytic codimension-2 CR singular submanifold
can be put, via a linear change of
coordinates, into the form \eqref{eq:formofM}, and therefore the
condition \eqref{eq:theABcondition} is the only hypothesis in the theorem.
As we shall see below, the condition \eqref{eq:theABcondition} is
necessary and sufficient for the quadric models.
In other words, the condition \eqref{eq:theABcondition} is the most
general hypothesis when considering only the quadratic part of $\rho$.

We are interested in those models $w = Q(z,\bar{z})$
that are CR singular and thus not complex manifolds.
That is, we are interested in those models where $Q$ depends
on $\bar{z}$, or equivalently, when
$A$ and $B$ are not both zero, so when
$\rank \left[ \begin{smallmatrix} A^* \\ B \end{smallmatrix} \right] \geq 1$.
We express 
this condition as $\bar{\partial} Q \not\equiv 0$.

As we noted, \eqref{eq:theABcondition} is
a necessary and sufficient condition
for extension on CR singular quadric models.
In fact, on the quadric models
the condition is equivalent to checking extension
for real-linear functions of $z$ (and $\bar{z}$).  More precisely,
we have the following theorem, which
is in fact the key step in the proof of the result above.
   
\begin{thm}\label{thm:PolyExtn}
Let $M \subset \C^{n+1}$, $n \geq 2$, be a quadric given
in coordinates $(z,w) \in \C^n \times \C$
by
\begin{equation}
w = Q(z,\bar{z}) =
z^* A z + \overline{z^t B z} + z^t C z ,
\end{equation}
where
$A,B,C$ are complex $n \times n$ matrices with $B$ and $C$ symmetric.
Assume that $\bar{\partial} Q \not\equiv 0$.

The following are equivalent.
 \begin{enumerate}[(a)]
  \item
   $\rank \begin{bmatrix} A^* \\ B \end{bmatrix} \geq 2$.
  \item
   Suppose $f(z,\bar{z})$ is a polynomial that, when considered as a
   function on $M$ (parametrized by $z$), is a CR function on $M_{\mathit{CR}}$.
   Then there exists a unique holomorphic polynomial $F(z,w)$ such that $f$ and $F$
   agree on $M$, that~is,
   \begin{equation}
     f(z,\bar{z}) = F\bigl(z, Q(z,\bar{z}) \bigr) .
   \end{equation}
   The weighted degrees of $F$ and $f$ are the same if $w$ is given weight
   two.  If $f$ is homogeneous, then $F$ is weighted homogeneous of same
   degree.
  \item
   Suppose $h(z,\bar{z})$ is a linear function that, when considered as a
   function on $M$ (parametrized by $z$), is a CR function on $M_{\mathit{CR}}$.
   Then $h$ is holomorphic, that is, it is independent of
   $\bar{z}_1,\ldots,\bar{z}_{n}$.
  \item $M$ is \emph{not} biholomorphically equivalent to one of the following
   (mutually inequivalent) exceptional cases:
   \begin{enumerate}[(1)]
     \item $w = \bar{z}_1 z_2 + \bar{z}_1^2$,
     \item $w = \bar{z}_1 z_2$,
     \item $w = \sabs{z_1}^2 + a \bar{z}_1^2$, $a \geq 0$,
     \item $w = \bar{z}_1^2$.
   \end{enumerate}
 \end{enumerate}
\end{thm}

\begin{remark} \label{remark:LFimages}
The four exceptional cases are precisely the CR singular quadrics that are
locally diffeomorphic images of $\R^2 \times \C^{n-1}$ via a CR map.
That is, write
$(z,w) = (z_1,z',w)$, and the quadric as
$w = Q(z_1,z',\bar{z}_1)$.  The map is
$(s,t,\xi) \mapsto \bigl(s+it,\xi,Q(s+it,\xi,s-it)\bigr)$.  Such submanifolds
were studied, including their nonextension property, in \cite{LMSSZ}.
The four cases are Levi-flat at CR points, but they are not the only Levi-flat
submanifolds.  Levi-flat CR singular submanifolds of codimension 2 were
classified in \cite{GongLebl}, and in particular the Levi-flat quadrics
for which extension holds are the precisely the
submanifolds equivalent to
$w = \bar{z}_1^2 + \cdots + \bar{z}_k^2$ for $2 \leq k \leq n$.
\end{remark}

A related consequence of Theorems~\ref{thm:mainext} and
\ref{thm:PolyExtn} is
the following.  Consider 
a real-analytic CR singular codimension-2 
submanifold
$M \subset \C^{n+1}$
whose CR structure extends through the
CR singular points, or equivalently $M$ is locally an immersion
of a generic (CR) manifold via a CR map.
Note that the
extended CR structure is the CR Nash
blowup, see Garrity~\cite{Garrity}.
Such $M$ do not have the extension property by \cite{LMSSZ}.
Let $M$ be given in the coordinates of Theorem~\ref{thm:mainext}.
Then
$\rank \left[ \begin{smallmatrix} A^* \\ B \end{smallmatrix} \right] \leq 1$, or in other words, the quadric model for $M$
is equivalent to either
one of those listed in part (d) of Theorem~\ref{thm:PolyExtn}
or $w=0$.  See Corollary~\ref{cor:CRimages}.

With regard to part (c) of Theorem~\ref{thm:PolyExtn},
consider a quadric $w=Q(z,\bar{z})$ with
$\bar{\partial} Q\not\equiv 0$, but where the rank condition fails.
Proposition~\ref{prop:LinExtnFails2} shows that there exists a vector $v$
such that $h = v \cdot \bar{z}$ is a CR function that is not a restriction
of a holomorphic function.  For example, when $n=2$ and
$\rank \left[ \begin{smallmatrix} A^* \\ B \end{smallmatrix} \right] = 1$, then the nullspace of $A^*$ and $B$ have a common
vector
$(v_1,v_2)$.
Then we may take
\begin{equation}
h(z,\bar{z}) = \bar{v}_2 \bar{z}_1 - \bar{v}_1 \bar{z}_2 .
\end{equation}

We remark that if $M$ is given by $w = Q+E$ and either 
$\bar{\partial} Q\equiv 0$ ($\rank \left[ \begin{smallmatrix} A^* \\ B \end{smallmatrix} \right] = 0$) or
$\rank \left[ \begin{smallmatrix} A^* \\ B \end{smallmatrix} \right] = 1$, then
the higher-order terms $E$ may determine whether an extension result holds or not.
When $\bar{\partial} Q\equiv 0$,
first consider the model $w=0$.
Then every
CR function is trivially the restriction of a holomorphic function.
On the other hand, if $M$ is given by $w = \snorm{z}^4$, then $f(z,\bar{z}) = \snorm{z}^2$ is a 
CR function that does not extend to a holomorphic function. See Example~\ref{ex:qzero}.
Now consider the case where
$\rank \left[ \begin{smallmatrix} A^* \\ B \end{smallmatrix} \right] = 1$.
Extension does not hold for the manifold $w=\bar{z}_1z_2$,
but it does hold for $w=\bar{z}_1z_2 + \bar{z}_2^3$.
See Example~\ref{ex:ehot}.

We use a similar general idea as the standard proof of Severi's
theorem, that is,
real-analytic CR functions on real-analytic CR submanifolds extend.
We complexify, then eliminate one (or more)
of the barred variables (depending on dimension), and then try to
eliminate the rest of the barred variables with the CR condition.
The issue here is that the CR
singular equation $w = \rho(z,\bar{z})$ together with
$\bar{w} = \bar{\rho}(\bar{z},z)$ in $\C^{n+1}$
only naturally eliminates one barred
variable, and then we are left with only $n-1$ CR vector fields to get rid of
$n$ other barred variables.  On these grounds it is clear that the extension
cannot hold in general.  The extension works only if we can in some manner
``solve for another barred variable'' in $w = \rho(z,\bar{z})$,
which cannot be done in general---and even if it can be done, we get only a
multivalued solution.

To illustrate the difficulty, consider 
the submanifold of $\C^3$ given by $w = z_1^2+z_2^2+\bar{z}_1^2+\bar{z}_2^2$.  We can
solve for $\bar{w} = \bar{z}_1^2+\bar{z}_2^2+ z_1^2+z_2^2$, but solving
for say $\bar{z}_1$, we find the multivalued 
$\pm \sqrt{w - z_1^2-z_2^2- \bar{z}_2^2}$.  One cannot just plug that in.
We can, however, use it to get rid of
all terms in $f$ where $\bar{z}_1$ is of any power higher than 1.  Then one
needs to use the CR vector field to get rid of not only $\bar{z}_2$
but also this first power of $\bar{z}_1$.

At the other end of the spectrum of difficulties is
the submanifold
$w = z_1\bar{z}_1 + z_2 \bar{z}_2$.  In this case, when we try to solve for
$\bar{z}_1$, we get negative powers of $z_1$.  In both cases, the question
seems to boil down to whether $\bar{z}_1$, or in general some linear
function of $\bar{z}_1$ and $\bar{z}_2$, is CR on $M$, and (as we see in the
theorem) that is in fact sufficient.

Let us outline the organization of this paper.
In \S\ref{sec:prelim}, we explain some basic notation and
preliminaries.
In \S\ref{sec:examples}, we provide some concrete
examples of the extension and nonextension phenomena.
In \S\ref{sec:quadrics}, we prove Theorem~\ref{thm:PolyExtn},
that is, the necessary  and sufficient conditions for
the extension phenomenon for polynomials on quadric models.
In \S\ref{sec:real-analytic}, we prove the main result, Theorem~\ref{thm:mainext}.
In \S\ref{sec:flattening}, we provide as an application a
connection of the extension result to flattening: on manifolds
where extension holds, flattening is equivalent to the existence of a
first integral for the singular foliation given by the CR orbits.
Finally, in \S\ref{sec:CRimages}, we use the results to
classify images of CR submanifolds up to the quadratic part.

We would like to thank and acknowledge Adam Coffman for useful
discussion about his
result~\cite{Coffman}, which proved invaluable in this work.
We would also like to thank the referee for useful suggestions.

%%%%%%%%%%%%%%%%%%%%%%%%%%%%%%%%%%%%%%%%%%%%%%%%%%%%%%%%%%%%%%%%%%%%%%%%%%%%

\section{Preliminaries} \label{sec:prelim}

Any real-analytic 
submanifold $M \subset \C^{n+1}$ of real codimension 2 (real dimension $2n$)
with a CR singularity at the origin can
locally, after a rotation by a unitary,
be represented in coordinates $(z,w) \in \C^2 \times \C$ by
the equation
\begin{equation} \label{eq:defM}
w = \rho(z,\bar{z})
\end{equation}
for a real-analytic function $\rho$ that is $O(\snorm{z}^2)$.
Let $M_{\mathit{CR}} \subset M$ denote the CR points of $M$, that is, the points
near which
\begin{equation}
T_p^{0,1}M = \C \otimes T_pM \cap \operatorname{span}_{\C} \left
\{
\frac{\partial}{\partial\bar{z}_1}\Big|_{p},
\ldots ,
\frac{\partial}{\partial\bar{z}_n}\Big|_{p},
\frac{\partial}{\partial\bar{w}}\Big|_{p}
\right\}
\end{equation}
is of constant dimension as $p \in M$ varies.
Because $M$ is not a complex manifold,
 $T_p^{0,1}M$ is of complex dimension $n-1$ at CR points.
Therefore, among the \emph{CR vector fields}, that is, vector fields valued
in $T_p^{0,1}M$, $n-1$ vector fields suffice to form
a basis for $T_p^{0,1}M$ at all CR points.  A CR vector field will
generally vanish at the CR singular points although at CR singular
points the dimension of $T_p^{0,1}M$ is $n$.

We say a real-analytic function $f$ defined on $M$ is a
\emph{CR function} if $Lf = 0$ for any CR vector field $L$ on $M$.
This definition is equivalent to saying that $f$ is a CR function on $M_{\mathit{CR}}$.

Extrinsically, a CR vector field can be written as (given  $1 \leq k,\ell
\leq n$)
\begin{equation}
\rho_{\bar{z}_\ell} \frac{\partial}{\partial \bar{z}_k} - 
\rho_{\bar{z}_k} \frac{\partial}{\partial \bar{z}_\ell} +
\bigl(
  \rho_{\bar{z}_\ell} \bar{\rho}_{\bar{z}_k} -
  \rho_{\bar{z}_k} \bar{\rho}_{\bar{z}_\ell}
\bigr)
\frac{\partial}{\partial \bar{w}} ,
\end{equation}
as it needs to kill both
$-w+ \rho$ and $-\bar{w}+ \bar{\rho}$.

Since $M$ is written as a graph over $z$, then we can use $z$ for
parameters.  Therefore, we write any function $f$ on $M$
as a function of $z$ and $\bar{z}$.
When we write the vector field intrinsically using
these parameters on $M$, we find:

\begin{prop} \label{prop:formofCRvecs}
Let $M$ be given by \eqref{eq:defM}, and let $f(z,\bar{z})$ be a real-analytic
function.  If we consider $f$ as a function on $M$, then $f$ is CR
if and only if $L_{k,\ell}f = 0$ for the vector fields
\begin{equation}
L_{k,\ell} =
\rho_{\bar{z}_\ell} \frac{\partial}{\partial \bar{z}_k} - 
\rho_{\bar{z}_k} \frac{\partial}{\partial \bar{z}_\ell} 
\qquad
\text{for $1 \leq k,\ell \leq n$.}
\end{equation}
\end{prop}

The set of CR singularities is precisely
where
$\rho_{\bar{z}_1} = \cdots = \rho_{\bar{z}_n} = 0$.  So $L_{k,\ell}$
all vanish precisely on the set of CR singularities of $M$, and where the
vector fields do not vanish, they span $T_p^{0,1} M$.

The question we are trying to answer is the following:
When is a function $f(z,\bar{z})$, as
a function on $M$, a restriction to $M$ of a holomorphic function $F(z,w)$?
A priori, to write the complexified equation for $f$ and $F$ to be equal on
$M$, we must complexify the equation $w = \rho(z,\bar{z})$ and its
conjugate.  That is, $f$ and $F$ are equal on $M$ near the origin
if there exist convergent power series $a$ and $b$ such that
\begin{equation}
F(z,w) = f(z,\bar{z}) +
a(z,\bar{z},w,\bar{w})
\bigl(
w - \rho(z,\bar{z})
\bigr)
+
b(z,\bar{z},w,\bar{w})
\bigl(
\bar{w} - \bar{\rho}(\bar{z},z)
\bigr) .
\end{equation}
This equation holds for all $z$ and $w$, and therefore we can treat
$\bar{z}$ and $\bar{w}$ as independent variables.  In other words,
we may replace $\bar{w}$ with
$\bar{\rho}(\bar{z},z)$ to remove the last term on the right.
Hence, we have the following proposition.
For later use we state the proposition
in somewhat greater generality,
where the functions are only required to be holomorphic
in $w$.

\begin{prop}\label{prop:Complexification}
Let $M$ be given by \eqref{eq:defM}, and let $\Phi(z,\bar{z},w)$
and $\Psi(z,\bar{z},w)$
be  real-analytic functions near the origin holomorphic in $w$.
Then $\Phi$ and $\Psi$ are equal on $M$ near the origin if and only if
\begin{equation}
\Phi(z,\bar{z},w) = \Psi(z,\bar{z},w) +
c(z,\bar{z},w)
\bigl(
w - \rho(z,\bar{z})
\bigr) 
\end{equation}
for some convergent power series $c(z,\bar{z},w)$.
In particular, when we complexify expressions that are holomorphic in $w$,
we only need to consider the
variables $\bar{z}$ variables as independent,
that is, we work in $\C^n \times \C^n \times \C$ using
independent variables $(z,\bar{z},w)$.
The complexified manifold is then the complex hypersurface
in $\C^{2n+1}$
given by $w = \rho(z,\bar{z})$.
\end{prop}

We can write the condition that $f(z,\bar{z})$ is equal to $F(z,w)$
on $M$
as either 
\begin{equation}
F(z,w) = f(z,\bar{z}) +
c(z,\bar{z},w,\bar{w})
\bigl(
w - \rho(z,\bar{z})
\bigr)
\qquad 
\text{or}
\qquad
F\bigl(z,\rho(z,\bar{z})\bigr) = f(z,\bar{z}).
\end{equation}
One advantage of the division form of the condition is that it is easier to
interpret with formal power series $F$.  With the second form, we have to first
note that $\rho$ has no constant terms and 
$F\bigl(z,\rho(z,\bar{z})\bigr)$ can be correctly interpreted formally.

A final note is that it is possible to absorb holomorphic terms from
$\rho$ into $w$ via a biholomorphic change of variables.  In particular,
the equation for a manifold given by
\begin{equation}
w = \rho(z,\bar{z}) =
Q(z,\bar{z}) + E(z,\bar{z}) =
z^*Az + \overline{z^tBz} + z^tCz + E(z, \bar{z}),
\end{equation}
can be changed by a quadratic change of coordinates such that either $C=0$
or perhaps $C=B$, and at times this normalization is useful and commonly
made.  However, in this work the form $z^tCz$ does not appear in any
essential way, and for the purposes of examples, $C$ can generally be
ignored.  This means that the only normalization we make
is a rotation via a unitary map in $(z,w)$,
and then perhaps a complex linear map in the $z$ variables.

%%%%%%%%%%%%%%%%%%%%%%%%%%%%%%%%%%%%%%%%%%%%%%%%%%%%%%%%%%%%%%%%%%%%%%%%%%%%

\section{Examples} \label{sec:examples}

Our first two examples show how extension of CR functions can fail.

\begin{example}\label{ex:qzero}
Suppose
that $Q \equiv 0$.  The model $w = Q(z,\bar{z}) = 0$ 
is a complex hypersurface, and in fact not even CR singular.  Any CR
function on $\{ w=0 \}$ is a holomorphic function in $z$ and so clearly
extends, although not uniquely.  However, consider the submanifold $M$ in $\C^{n+1}$ given by
\begin{equation}
w = {( \sabs{z_1}^2 + \sabs{z_2}^2 + \cdots + \sabs{z_n}^2)}^2 = \snorm{z}^4 .
\end{equation}
This $M$ is CR singular, with the origin being an isolated CR singularity.
The function $f(z,\bar{z}) = \snorm{z}^2$
is
CR because, on $M \setminus \{ 0 \}$,
the function is equal to one of the branches of $\sqrt{w}$,
which is holomorphic.
Hence, $f$ is CR outside the origin as needed.
Because $M$ is a
generic manifold outside the origin, the holomorphic extension is
unique near every such point.  That means any possible extension at the
origin would have to equal $\sqrt{w}$ on an open set, which is preposterous.
\end{example} 

\begin{example}\label{ex:LinExtnFails}
Consider the quadric submanifold in $\C^3$ given by
\begin{equation}
w = Q(z,\bar{z})=\bar{z}_1 z_2.
\end{equation}
The CR vector field is given by
\begin{equation}
L  = 
Q_{\bar{z}_2} \frac{\partial}{\partial \bar{z}_1}
-
Q_{\bar{z}_1} \frac{\partial}{\partial \bar{z}_2}=-z_2 \frac{\partial}{\partial \bar{z}_2}.
\end{equation}
Thus, $f(z,\bar{z})=\bar{z}_1$
is a CR function on the quadric.
However, $f$ cannot be equal to any holomorphic function in a
neighborhood of the origin.  If it were, the unique holomorphic
extension would equal $w/z_2$ on an open set, a contradiction.
\end{example}

It is entirely possible that extension fails on a quadric model for $M$,
but the extension result holds on $M$ due to higher-order terms.

\begin{example}\label{ex:ehot}
The extension result does not hold for the submanifold in $\C^3$ given by
\begin{equation}
w = \bar{z}_1 z_2 .
\end{equation}
We have seen above that $f = \bar{z}_1$ is an example of a CR function that
does not extend.

However, consider $M$ given by
\begin{equation}
w = \rho(z,\bar{z})=\bar{z}_1 z_2 + \bar{z}_2^3 .
\end{equation}
We claim that any real-analytic function that is CR on $M_{\mathit{CR}}$ extends to
a holomorphic function near the origin. 

Let $f(z,\bar{z})$ be a real-analytic CR function written in terms of the
$z,\bar{z}$ variables.
As $f$ is CR, we have $Lf = 0$, where $L$ is the CR vector field:
\begin{equation}
L =
\rho_{\bar{z}_2} \frac{\partial}{\partial \bar{z}_1}
-
\rho_{\bar{z}_1} \frac{\partial}{\partial \bar{z}_2}=
3\bar{z}_2^2 \frac{\partial}{\partial \bar{z}_1}
-
z_2 \frac{\partial}{\partial \bar{z}_2} .
\end{equation}

Since $f$ is a function of $z$ and $\bar{z}$, it does not
depend on $\bar{w}$; so, when we complexify we work in the variables
$z_1,z_2,\bar{z}_1,\bar{z}_2,w$ and we treat the barred variables
as independent.
See Proposition~\ref{prop:Complexification}.
We consider $f$ to be a function on the
complex submanifold $\widetilde{M}$ in $\C^5$ given by
$w = \rho(z,\bar{z})$ in this way.  The complexified
vector field $L$ is tangent to this manifold.

We apply the Weierstrass division algorithm to $f$ using the variable $\bar{z}_2$ and write
\begin{equation}
f(z,\bar{z}) = h(z,\bar{z},w) \bigl(\bar{z}_1z_2+\bar{z}_2^3-w\bigr)
+ a(z,\bar{z}_1,w) +
b(z,\bar{z}_1,w) \bar{z}_2 +  c(z,\bar{z}_1,w)\bar{z}_2^2.
\end{equation}
Then $f$ equals a holomorphic function restricted to $M$ if 
$a_{\bar{z}_1}\equiv 0$, $b \equiv 0$, and $c \equiv  0$.

Since $L$ kills both $f$ and $\bar{z}_1z_2+\bar{z}_2^3-w$, on $\widetilde{M}$ we have
\begin{equation}
L\bigl(
a(z,\bar{z}_1,w) +
b(z,\bar{z}_1,w) \bar{z}_2+c(z,\bar{z}_1,w) \bar{z}_2^2 \bigr) = 0 .
\end{equation}
Thus, on $\widetilde{M}$ we have
\begin{equation}
\begin{split}
0 & = L\bigl(
a +
b \bar{z}_2  +c \bar{z}_2^2 \bigr)
\\
& =
-b z_2 -
2c z_2 \bar{z}_2 + 3 a_{\bar{z}_1} \bar{z}_2^2 +3b_{\bar{z}_1}\bar{z}_2^3+3c_{\bar{z}_1}\bar{z}_2^4 .
\end{split}
\end{equation}
We substitute for $\bar{z}_1^3$ using the defining equation for
$\widetilde{M}$ and find
\begin{multline}
-b z_2 -
2c z_2 \bar{z}_2 + 3 a_{\bar{z}_1} \bar{z}_2^2 +3b_{\bar{z}_1}(w-\bar{z}_1z_2)+
3c_{\bar{z}_1}(w-\bar{z}_1z_2)\bar{z}_2\\
=
\bigl(-b z_2 +3b_{\bar{z}_1}(w-\bar{z}_1z_2)\bigr)
+\bigl(-2c z_2+3c_{\bar{z}_1}(w-\bar{z}_1z_2) \bigr)\bar{z}_2 + 3 a_{\bar{z}_1} \bar{z}_2^2.
\end{multline}
This expression is zero on $\widetilde{M}$, so it is divisible by
$\bar{z}_1z_2+\bar{z}_2^3-w$, which is a Weierstrass polynomial in $\bar{z}_2$ of degree 3.
Hence, the expression is identically zero.
We get three equations:
\begin{align*}
& b z_2 =3b_{\bar{z}_1}(w-\bar{z}_1z_2) , \\
& 2c z_2=3c_{\bar{z}_1}(w-\bar{z}_1z_2) , \\
& 3 a_{\bar{z}_1}=0.
\end{align*}
Thinking of $b$ as the dependent variable and $\bar{z}_1$ as the independent variable,
we solve the first differential equation explicitly
to get $b= \alpha (w-\bar{z}_1z_2)^{-1/3}$ on an open dense set of $z_1, z_2, w$
for some $\alpha$ independent of $\bar{z}_1$.
Thus, there is no nonzero real-analytic solution $b$ of the first equation.
Similarly, the solution of the second equation is  $c= \alpha (w-\bar{z}_1z_2)^{-2/3}$
on an open dense set of $z_1, z_2, w$.
Thus, there is no nonzero real-analytic solution $c$ of the second equation.
Using these three equations, 
we conclude that $a_{\bar{z}_1}\equiv 0$, $b \equiv 0$, and $c \equiv  0$, as desired.
Thus, $f$ has a holomorphic extension.
\end{example}

%%%%%%%%%%%%%%%%%%%%%%%%%%%%%%%%%%%%%%%%%%%%%%%%%%%%%%%%%%%%%%%%%%%%%%%%%%%%

\section{Extending CR polynomials on quadrics}\label{sec:quadrics}

%Proposition~\ref{prop:counter} is a consequence of Example \ref{ex:qzero}
%and the following result.

In this section we prove the extension result for CR polynomials on quadrics.
That is, we prove the equivalence of the conditions
(a), (b), (c), and (d) in Theorem~\ref{thm:PolyExtn}.
The following proposition proves that (c) implies (a) by contrapositive.
Note that if
$\bar{\partial}Q \not\equiv 0$, then at least one of $A$ or $B$
is not zero.
So if (a) is not true, then
$\rank \left[ \begin{smallmatrix} A^* \\ B \end{smallmatrix} \right] = 1$.

\begin{prop}\label{prop:LinExtnFails2}
Suppose $M \subset \C^{n+1}$, $n \geq 2$, is a quadric given
in coordinates $(z,w) \in \C^n \times \C$
by
\begin{equation}
w = Q(z,\bar{z}) =
z^* A z + \overline{z^t B z} + z^t C z ,
\end{equation}
where
$A,B,C$ are complex $n \times n$ matrices with $B$ and $C$ symmetric.
Suppose that
\begin{equation}
   \rank \begin{bmatrix} A^* \\ B \end{bmatrix} = 1 .
\end{equation}
Then there exists a vector
$v \in \C^n$ such that the
linear function $f(z,\bar{z}) = v \cdot \bar{z}$,
when considered as a function on $M$, is a CR function,
but on no neighborhood of the origin does there exist a holomorphic
function $F$ whose restriction to $M$ is $f$. 
\end{prop}

\begin{proof} 
First let us make a linear change of coordinates in $z$.  If the change of
coordinates matrix is $T$, then the matrix $A$
transforms as $T^* A T$, so $A^*$ transforms as 
$T^* A^* T$, and $B$ transforms as $T^t B T$.  Using the rank condition,
we choose an invertible $T$ such that $A^*T$ and $BT$ have zero entries in
all but the first column, and then so do the matrices $T^*A^*T$ and $T^tBT$.
Thus, without loss of generality, let us assume that both $A^*$ and $B$ have zero
entries in all but the first column.  As $B$ is symmetric, it is only the
top left entry in $B$ that may be nonzero.  Furthermore, only the first row of $A$ may be
nonzero.

Given the form of $A$ and $B$ and the
fact that
$\rank \left[ \begin{smallmatrix} A^* \\ B
\end{smallmatrix} \right] = 1$, we find that
$Q_{\bar{z}_1} \not= 0$, and
$Q_{\bar{z}_k} \equiv 0$ for $k=2,\ldots,n$.
It is not difficult to see that the vector fields
\begin{equation}
L_k = \frac{\partial}{\partial \bar{z}_k}  \qquad k=2,\ldots,n
\end{equation}
are CR vector fields (intrinsically, using $z$ as a parameter)
that span $T^{0,1} M$ at CR points.

Therefore, the function $f = \bar{z}_1$ is a CR function.
However, $f$ cannot be equal to any holomorphic function $F$ in a
neighborhood of the origin:  If it were,
we would have (see Proposition~\ref{prop:Complexification})
\begin{equation}
F(z,w) =
\bar{z}_1
+ a(z,\bar{z},w)\bigl(w - Q(z,\bar{z})\bigr)
\end{equation}
for some convergent power series $a$; but that is impossible since
$a(z,\bar{z},w)\bigl(w - Q(z,\bar{z})\bigr)$ has no linear terms in
$\bar{z}$.
\end{proof}

Next, let us prove that (a) is equivalent with (d).

\begin{prop} \label{prop:equivofAandD}
Suppose $M \subset \C^{n+1}$, $n \geq 2$, is a quadric given
in coordinates $(z,w) \in \C^n \times \C$
by
\begin{equation}
w = Q(z,\bar{z}) =
z^* A z + \overline{z^t B z} + z^t C z ,
\end{equation}
where
$A,B,C$ are complex $n \times n$ matrices with $B$ and $C$ symmetric.
Then
\begin{equation}
   \rank \begin{bmatrix} A^* \\ B \end{bmatrix} = 1
\end{equation}
if and only if $M$ is biholomorphically equivalent to one of the following
quadrics:
   \begin{enumerate}[(1)]
     \item $w = \bar{z}_1 z_2 + \bar{z}_1^2$,
     \item $w = \bar{z}_1 z_2$,
     \item $w = \sabs{z_1}^2 + a \bar{z}_1^2$, $a \geq 0$,
     \item $w = \bar{z}_1^2$.
   \end{enumerate}
\end{prop}

\begin{proof}
One direction is immediate.
For the other direction, assume the rank condition.
First, we dispose of $C$ by folding holomorphic terms into $w$.
Using the same argument as in the previous proposition, we first normalize
$A$ to have all but the first row zero and $B$ to have all but the top
left entry zero.  Now we use a transformation that touches only
$z_2$ through $z_n$, that is, we use a matrix of the form $T = [1] \oplus T'$,
which leaves $A$ and $B$ in the form described above. Using such a matrix, we can ensure that the
first row of $A$ has zeros in all but the first two entries.
In other words, we have reduced the problem a problem in $\C^3$,
and the normal form can be now computed directly.  Or alternatively,
we refer to the work of Coffman~\cite{Coffman},
who classified the quadrics in $\C^3$.
\end{proof}

Next we show that, in the setup of Theorem~\ref{thm:PolyExtn},
(a) implies (b). This implication is proved using separate lemmas to cover the two cases $B\neq 0$ and $B=0$. For the proof of the first lemma  we need the following elementary result.

\begin{prop} \label{diffyqs:prop}
Assume that $p$, $q$, $r$, $s$, and $t$ are constants.
\begin{enumerate}[(a)]
\item
Consider $(p+q \eta)\zeta = (r+s \eta)\frac{d\zeta}{d\eta}$, $s \not= 0$.
The equation has a nonconstant polynomial solution
if and only if $q=0$ and $\frac{p}{s}$ is a positive integer.
The equation has a nonzero constant solution if and only if
$p=q=0$.
\item
Consider $(p+q \eta)\zeta = (r+s \eta+t \eta^2)\frac{d\zeta}{d\eta}$, $t \not=0$.  Let $\xi_1$ and
$\xi_2$ be the two roots of $r+s \eta+t \eta^2$, and suppose $\xi_1
\not=\xi_2$.  Then
the equation has a nonconstant polynomial solution
if and only if
$\frac{q\xi_1+p}{t(\xi_1-\xi_2)}$ and
$\frac{q\xi_2+p}{t(\xi_2-\xi_1)}$ are both nonnegative integers,
and one of them is positive.
The equation has a nonzero constant solution if and only if
$q \xi_1 + p = q \xi_2 + p =0$.
\item
The equation $(p+q \eta)\zeta = t {(\eta-\xi)}^2\frac{d\zeta}{d\eta}$ for $t\not=0$,
has a nonconstant polynomial solution
if and only if $\frac{q}{t}$ is a positive
integer and $q\xi+p=0$.
The equation has a nonzero constant solution if and only if
$q=p=0$.
\end{enumerate}
\end{prop}

\begin{proof}
The assertions follow by explicitly solving the equations: for some constant $\alpha$,
\begin{align*}
& \text{(a)} \qquad \zeta =
\alpha {(s\eta+r)}^{\frac{ps-qr}{s^2}}
e^{\frac{q \eta}{s}} ,
\\
& \text{(b)} \qquad \zeta =
\alpha
{(\eta-\xi_1)}^{\frac{q\xi_1+p}{t(\xi_1-\xi_2)}}
{(\eta-\xi_2)}^{\frac{q\xi_2+p}{t(\xi_2-\xi_1)}} ,
\\
& \text{(c)} \qquad \zeta =
\alpha {(\eta-\xi)}^{\frac{q}{t}} e^{\frac{-(q \xi + p)}{t(\eta-\xi)}} .
\qedhere
\end{align*}
\end{proof}

\begin{lemma}\label{lem:PolyExtnDE}
In the setup of Theorem \ref{thm:PolyExtn}, if $n=2$ and $B\neq 0$ then (a) implies (b).
\end{lemma}

\begin{proof}
Let $f(z,\bar{z})$ be a polynomial CR function written in terms of the
$z,\bar{z}$ variables.  The CR vector field
\begin{equation}
L =
Q_{\bar{z}_2} \frac{\partial}{\partial \bar{z}_1}
-
Q_{\bar{z}_1} \frac{\partial}{\partial \bar{z}_2}
\end{equation}
is a CR vector field defined intrinsically on $M$ in the $z$ variables.
As $f$ is CR, we have $Lf = 0$.

Because $B\neq 0$, by a linear change of coordinates in the
$z$ variables the symmetric matrix $B$ can be put into the
form
$\left[ \begin{smallmatrix} 1 & 0 \\ 0 & \epsilon \end{smallmatrix} \right]$
for some $\epsilon = 0,1$.  In particular,
\begin{equation}
Q(z,\bar{z}) = z^*Az +
\bar{z}_1^2 + 
\epsilon \bar{z}_2^2 + z^tCz .
\end{equation}

Since $f$ is a function of $z$ and $\bar{z}$, it does not
depend on $\bar{w}$; so, when we complexify we work in the variables
$z_1,z_2,\bar{z}_1,\bar{z}_2,w$, and we treat the barred variables
as independent.
See Proposition~\ref{prop:Complexification}.
We consider $f$ to be a function on the complex submanifold $\widetilde{M}$ in $\C^5$ given by
$w = Q(z,\bar{z})$.  The complexified vector field $L$
is tangent to this manifold.

We apply the Weierstrass division algorithm to $f$ using the variable $\bar{z}_1$
 and write
\begin{equation}
f(z,\bar{z}) = h(z,\bar{z},w) \bigl(Q(z,\bar{z})-w\bigr)
+ a(z,\bar{z}_2,w) +
b(z,\bar{z}_2,w) \bar{z}_1 .
\end{equation}
The polynomial $f$ equals a holomorphic polynomial $g(z,w)$ on $M$
if and only if
\begin{equation}
a(z,\bar{z}_2,w) +
b(z,\bar{z}_2,w) \bar{z}_1 - g(z,w)
\end{equation}
is divisible by $Q-w$, which is of degree 2 in $\bar{z}_1$, and hence
if and only if $a+b \bar{z}_1-g$ is identically zero.
In other words, $f$ is equal to a holomorphic polynomial $g$ on 
$M$ if and only if $a_{\bar{z}_2}\equiv 0$ and $b \equiv 0$.

Since $L$ kills both $f$ and $Q-w$, on $\widetilde{M}$ we have
\begin{equation}
L\bigl(
a(z,\bar{z}_2,w) +
b(z,\bar{z}_2,w) \bar{z}_1 \bigr) = 0 .
\end{equation}
Let us write as a matrix $A = 
\left[
\begin{smallmatrix} \alpha & \beta \\ \gamma & \delta \end{smallmatrix}
\right]$ and
$L$ as
\begin{equation}
L =
\bigl(
\gamma z_1 + \delta z_2 +
2 \epsilon \bar{z}_2 \bigr)
\frac{\partial}{\partial \bar{z}_1}
-
\bigl(\alpha z_1 + \beta z_2 + 2 \bar{z}_1\bigr)
\frac{\partial}{\partial \bar{z}_2} .
\end{equation}
Then on $\widetilde{M}$ we have
\begin{equation}
\begin{split}
0 & = L(
a +
b \bar{z}_1 )
\\
& =
\bigl(
\gamma z_1 + \delta z_2 +
2 \epsilon \bar{z}_2 \bigr) b
-
\bigl(\alpha z_1 + \beta z_2 + 2 \bar{z}_1\bigr)
\bigl(
a_{\bar{z}_2} +
b_{\bar{z}_2} \bar{z}_1 \bigr) 
\\
& =
\Bigl(\bigl(
\gamma z_1 + \delta z_2 +
2 \epsilon \bar{z}_2 \bigr) b
-
(\alpha z_1 + \beta z_2)
a_{\bar{z}_2}
\Bigr)
-
\bigl(
2
a_{\bar{z}_2}
+
(\alpha z_1 + \beta z_2)
b_{\bar{z}_2}
\bigr) \bar{z}_1
-
2
b_{\bar{z}_2} \bar{z}_1^2 .
\end{split}
\end{equation}
As this is on $\widetilde{M}$, let us substitute for $\bar{z}_1^2$ in this expression
using the defining equation for $\widetilde{M}$.
We find 
\begin{multline}
\Bigl(\bigl(
\gamma z_1 + \delta z_2 +
2 \epsilon \bar{z}_2 \bigr) b
-
(\alpha z_1 + \beta z_2)
a_{\bar{z}_2}
\Bigr)
\\
-
\bigl(
2 a_{\bar{z}_2}
+
(\alpha z_1 + \beta z_2)
b_{\bar{z}_2}
\bigr) \bar{z}_1
-
2
b_{\bar{z}_2}
\Bigl(
w - (\alpha z_1 + \beta z_2) \bar{z}_1
-
(\gamma z_1 + \delta z_2) \bar{z}_2 - 
( z^tCz + \epsilon \bar{z}_2^2 )
\Bigr)
\\
=
\Bigl(\bigl(
\gamma z_1 + \delta z_2 +
2 \epsilon \bar{z}_2 \bigr) b
-
(\alpha z_1 + \beta z_2)
a_{\bar{z}_2}
-
2
\bigl(w - 
(\gamma z_1 + \delta z_2) \bar{z}_2 - 
( z^tCz + \epsilon \bar{z}_2^2 )
\bigr)
b_{\bar{z}_2}
\Bigr)
\\
+
\bigl(
- 2 a_{\bar{z}_2}
+
(\alpha z_1 + \beta z_2)
b_{\bar{z}_2}
\bigr)
\bar{z}_1 .
\end{multline}
This expression is zero on $\widetilde{M}$, so it is divisible by
$Q-w$, which is a Weierstrass polynomial in $\bar{z}_1$ of degree 2.  Hence,
the expression is identically zero.
In particular,
\begin{equation}
2 a_{\bar{z}_2} = (\alpha z_1 + \beta z_2) b_{\bar{z}_2} ,
\end{equation}
and
\begin{multline}
0=\bigl(
\gamma z_1 + \delta z_2 +
2 \epsilon \bar{z}_2 \bigr) b
-
(\alpha z_1 + \beta z_2)
a_{\bar{z}_2}
-
2
\bigl(w - 
(\gamma z_1 + \delta z_2) \bar{z}_2 - 
( z^tCz + \epsilon \bar{z}_2^2 )
\bigr)
b_{\bar{z}_2}
\\
=
\bigl(
\gamma z_1 + \delta z_2 +
2 \epsilon \bar{z}_2 \bigr) b
-
\left(
\frac{(\alpha z_1 + \beta z_2)^2}{2}
+
2
\bigl(w - 
(\gamma z_1 + \delta z_2) \bar{z}_2 - 
( z^tCz + \epsilon \bar{z}_2^2 )
\bigr)
\right)
b_{\bar{z}_2}
\\
=
\bigl(
(\gamma z_1 + \delta z_2) +
2 \epsilon \bar{z}_2 \bigr) b
-
\left(
\left(
\frac{(\alpha z_1 + \beta z_2)^2}{2}
+
2 w
- z^tCz
\right)
-
2
(\gamma z_1 + \delta z_2) \bar{z}_2
- 
2
\epsilon \bar{z}_2^2
\right)
b_{\bar{z}_2}.
\end{multline}

First suppose that $\epsilon = 0$. Then either $\gamma$ or $\delta$ is nonzero: otherwise, $\rank \left[ \begin{smallmatrix} A^* \\ B \end{smallmatrix} \right] = 1$, contradicting (a) of Theorem \ref{thm:PolyExtn}. The equation is
\begin{equation}
( \gamma z_1 + \delta z_2 ) b
-
\left(
\left(
\frac{(\alpha z_1 + \beta z_2)^2}{2}
+
2 w
-z^tCz
\right)
-
2
(\gamma z_1 + \delta z_2) \bar{z}_2
\right)
b_{\bar{z}_2}
= 0 .
\end{equation}
On an open dense set of $z_1,z_2$, we have $\gamma z_1 +
\delta z_2 \not= 0$.
In the notation of part (a) of Proposition~\ref{diffyqs:prop},  $p = \gamma z_1 +
\delta z_2$, $q = 0$, and $s = -2(\gamma z_1 + \delta z_2)$, so $\frac{p}{s} =
-\frac{1}{2}$ on an open dense set of $z_1,z_2$ (and any $w$).  By Proposition~\ref{diffyqs:prop},
there is no nonzero polynomial solution $b$.  Hence $b= 0$, so $a_{\bar{z}_2} = 0$.
Thus, $a$ is the holomorphic polynomial we are seeking.

Now suppose that $\epsilon = 1$. Again, we must prove that $b = 0$.  The equation is
\begin{equation}
\bigl(
(\gamma z_1 + \delta z_2) +
2 \bar{z}_2 \bigr) b
-
\left(
\left(
\frac{(\alpha z_1 + \beta z_2)^2}{2}
+
2 w
-z^tCz
\right)
-
2
(\gamma z_1 + \delta z_2) \bar{z}_2
- 
2
\bar{z}_2^2
\right)
b_{\bar{z}_2}
= 0 .
\end{equation}
In the notation of parts (b)--(c) of 
Proposition~\ref{diffyqs:prop}, in the polynomial $r+s \eta+t \eta^2$
the $r$ depends on $w$.  That is, for any fixed $z_1,z_2$, we can change $w$ by
a small amount to change the roots of the quadratic by a small amount in any direction.
In other words, we can always achieve 
that neither $\frac{q\xi_1+p}{t(\xi_1-\xi_2)}$ nor
$\frac{q\xi_2+p}{t(\xi_1-\xi_2)}$ is a nonnegative integer.
So, for an open dense set of $z_1,z_2,w$,  the only polynomial
solution is $b=0$. Hence, $b=0$.  Again, this means that $a_{\bar{z}_2} = 0$, and
so $a$ is the desired holomorphic polynomial.

The claim in (b) of Theorem \ref{thm:PolyExtn} about degrees now follows from the division algorithm.  Consider $f(z,\bar{z})$ as a
weighted homogeneous function of $(z,\bar{z},w)$.  As can be seen by walking through the algorithm,
the remainder $F(z,w)$ after division by $Q(z,\bar{z})-w$, which is weighted homogeneous,
must also be weighted homogeneous of the same degree as $f$.
\end{proof}

If $B=0$ then no barred variable appears in a pure term
in the quadratic part $Q$. Thus, the Weierstrass theorem cannot be applied, and the preceding method of proof is not available. We therefore use a different method when $B=0$.

\begin{lemma}\label{lem:PolyExtnMa}
In the setup of Theorem \ref{thm:PolyExtn},
if $n=2$ and $B= 0$, then (a) implies (b).
\end{lemma}

\begin{proof}
Because $B=0$, the rank condition in (a) implies that $A$ has rank 2.
In this case, we put $A$
into the form $\left[ \begin{smallmatrix} 1 & \beta \\ 0 & \delta
\end{smallmatrix} \right]$ with $\delta\neq 0$.  That is, we first apply
the unitary from 
Schur's theorem to the $z$ variables
to make $A$ triangular, and then we rescale the $w$ to make the top left entry
1.

We order the monomials in $(z,\bar{z})$ by an ordering that satisfies
$z_1^{a_1} z_2^{a_2}
\bar{z}_1^{b_1} \bar{z}_2^{b_2} <
z_1^{a_3} z_2^{a_4}
\bar{z}_1^{b_3} \bar{z}_2^{b_4}$
if $a_1+a_2 < a_3+a_4$ or if
$a_1+a_2 = a_3+a_4$ and
$a_1 < a_3$, or if
$a_1+a_2 = a_3+a_4$ and
$a_1 = a_3$ and
$b_1 < b_3$. 
We use the CR vector field \begin{equation}
L =
Q_{\bar{z}_2} \frac{\partial}{\partial \bar{z}_1}
-
Q_{\bar{z}_1} \frac{\partial}{\partial \bar{z}_2}.
\end{equation}
A homogeneous polynomial $f$ is CR if and only if $Lf = 0$, which is a linear equation in the
coefficients of $f$.  Using the preceding ordering on the monomials to order the coefficients, we let $c$ be the
vector of coefficients of $f$. For homogeneous polynomials of degree $d$, the equation $Lf = 0$ can be written as a matrix equation $X_dc = 0$.

\begin{claim}
The rank $R_d$ of $X_d$ satisfies 
\begin{equation}\label{eq:rank}
R_d=\sum_{j=1}^d 2\left\lfloor \frac{j+1}{2} \right\rfloor(d-j+1) .
\end{equation}
\end{claim}

\begin{proof}
Because $A = \left[ \begin{smallmatrix} 1 & \beta \\ 0 & \delta
\end{smallmatrix} \right]$, we have
$Q = z_1 \bar{z}_1 + \beta z_2 \bar{z}_1 + \delta z_2 \bar{z}_2+ z^tCz$.
Then 
\begin{equation}
L =
Q_{\bar{z}_2} \frac{\partial}{\partial \bar{z}_1}
-
Q_{\bar{z}_1} \frac{\partial}{\partial \bar{z}_2}
=
\delta z_2 \frac{\partial}{\partial \bar{z}_1}
-
( z_1 + \beta z_2) \frac{\partial}{\partial \bar{z}_2} .
\end{equation}
Therefore,
\begin{equation}\label{eq:homogL}
L \Bigl[ z_1^{a_1} z_2^{a_2} \bar{z}_1^{b_1} \bar{z}_2^{b_2} \Bigr]
=
b_1
\delta
\,
z_1^{a_1} z_2^{a_2+1} 
\bar{z}_1^{b_1-1} \bar{z}_2^{b_2}
-
b_2
\,
z_1^{a_1+1} z_2^{a_2} 
\bar{z}_1^{b_1} \bar{z}_2^{b_2-1}
-
b_2
\beta
\,
z_1^{a_1} z_2^{a_2+1} 
\bar{z}_1^{b_1} \bar{z}_2^{b_2-1} .
\end{equation}
For each input monomial, we get in general 3 output monomials, although
some can be zero depending on whether $b_1 = 0$ or $b_2 = 0$.  Thus, for the matrix $X_d$,
each column has 0, 1, 2, or 3 nonzero entries.

For example, for degree 3, the matrix $X_3$ is given by

{
\tiny
\begin{equation*}
\arraycolsep=1pt
\begin{array}{c||cccc|cccccc|cccccc|cccc}
 & \mathbin{\rotatebox[origin=c]{90}{$\bar{z}_1^{3}$}} & \mathbin{\rotatebox[origin=c]{90}{$\bar{z}_1^{2}\bar{z}_2$}} & \mathbin{\rotatebox[origin=c]{90}{$\bar{z}_1\bar{z}_2^{2}$}} & \mathbin{\rotatebox[origin=c]{90}{$\bar{z}_2^{3}$}} & \mathbin{\rotatebox[origin=c]{90}{$z_1\bar{z}_1^{2}$}} & \mathbin{\rotatebox[origin=c]{90}{$z_1\bar{z}_1\bar{z}_2$}} & \mathbin{\rotatebox[origin=c]{90}{$z_1\bar{z}_2^{2}$}} & \mathbin{\rotatebox[origin=c]{90}{$z_2\bar{z}_1^{2}$}} & \mathbin{\rotatebox[origin=c]{90}{$z_2\bar{z}_1\bar{z}_2$}} & \mathbin{\rotatebox[origin=c]{90}{$z_2\bar{z}_2^{2}$}} & \mathbin{\rotatebox[origin=c]{90}{$z_1^{2}\bar{z}_1$}} & \mathbin{\rotatebox[origin=c]{90}{$z_1^{2}\bar{z}_2$}} & \mathbin{\rotatebox[origin=c]{90}{$z_1z_2\bar{z}_1$}} & \mathbin{\rotatebox[origin=c]{90}{$z_1z_2\bar{z}_2$}} & \mathbin{\rotatebox[origin=c]{90}{$z_2^{2}\bar{z}_1$}} & \mathbin{\rotatebox[origin=c]{90}{$z_2^{2}\bar{z}_2$}} & \mathbin{\rotatebox[origin=c]{90}{$z_1^{3}$}} & \mathbin{\rotatebox[origin=c]{90}{$z_1^{2}z_2$}} & \mathbin{\rotatebox[origin=c]{90}{$z_1z_2^{2}$}} & \mathbin{\rotatebox[origin=c]{90}{$z_2^{3}$}}\\
\hline
\hline
\bar{z}_1^{3} & . & . & . & . & . & . & . & . & . & . & . & . & . & . & . & . & . & . & . & .\\
\bar{z}_1^{2}\bar{z}_2 & . & . & . & . & . & . & . & . & . & . & . & . & . & . & . & . & . & . & . & .\\
\bar{z}_1\bar{z}_2^{2} & . & . & . & . & . & . & . & . & . & . & . & . & . & . & . & . & . & . & . & .\\
\bar{z}_2^{3} & . & . & . & . & . & . & . & . & . & . & . & . & . & . & . & . & . & . & . & .\\
\hline
z_1\bar{z}_1^{2} & . & -1 & . & . & . & . & . & . & . & . & . & . & . & . & . & . & . & . & . & .\\
z_1\bar{z}_1\bar{z}_2 & . & . & -2 & . & . & . & . & . & . & . & . & . & . & . & . & . & . & . & . & .\\
z_1\bar{z}_2^{2} & . & . & . & -3 & . & . & . & . & . & . & . & . & . & . & . & . & . & . & . & .\\
z_2\bar{z}_1^{2} & 3\delta & -\beta & . & . & . & . & . & . & . & . & . & . & . & . & . & . & . & . & . & .\\
z_2\bar{z}_1\bar{z}_2 & . & 2\delta & -2\beta & . & . & . & . & . & . & . & . & . & . & . & . & . & . & . & . & .\\
z_2\bar{z}_2^{2} & . & . & \delta & -3\beta & . & . & . & . & . & . & . & . & . & . & . & . & . & . & . & .\\
\hline
z_1^{2}\bar{z}_1 & . & . & . & . & . & -1 & . & . & . & . & . & . & . & . & . & . & . & . & . & .\\
z_1^{2}\bar{z}_2 & . & . & . & . & . & . & -2 & . & . & . & . & . & . & . & . & . & . & . & . & .\\
z_1z_2\bar{z}_1 & . & . & . & . & 2\delta & -\beta & . & . & -1 & . & . & . & . & . & . & . & . & . & . & .\\
z_1z_2\bar{z}_2 & . & . & . & . & . & \delta & -2\beta & . & . & -2 & . & . & . & . & . & . & . & . & . & .\\
z_2^{2}\bar{z}_1 & . & . & . & . & . & . & . & 2\delta & -\beta & . & . & . & . & . & . & . & . & . & . & .\\
z_2^{2}\bar{z}_2 & . & . & . & . & . & . & . & . & \delta & -2\beta & . & . & . & . & . & . & . & . & . & .\\
\hline
z_1^{3} & . & . & . & . & . & . & . & . & . & . & . & -1 & . & . & . & . & . & . & . & .\\
z_1^{2}z_2 & . & . & . & . & . & . & . & . & . & . & \delta & -\beta & . & -1 & . & . & . & . & . & .\\
z_1z_2^{2} & . & . & . & . & . & . & . & . & . & . & . & . & \delta & -\beta & . & -1 & . & . & . & .\\
z_2^{3} & . & . & . & . & . & . & . & . & . & . & . & . & . & . & \delta & -\beta & . & . & . & .\\
\end{array}

\end{equation*}
}

The rows and columns are marked according to the corresponding monomial.
The zero entries are marked simply by dots for clarity.
We have also divided $X_d$ into blocks as follows.
Given
$j \geq 1$, consider the columns corresponding to monomials of total degree
$j$ in $\bar{z}$ and rows of total degree $j-1$ in $\bar{z}$.
Due to the form of $L$, the only monomials that
result from monomials of total degree $j$ in $\bar{z}$ are in fact
monomials of total degree $j-1$ in $\bar{z}$. We see that, except for the
zero rows and columns, the matrix has a direct sum structure with these
blocks as the summands.  Denote the $j$th block, the block
whose columns correspond monomials of degree $j$ in $\bar{z}$,
by $B_j$.
Let $r(j,d)$ denote the rank of $B_j$.
The rank $R_d$ is the sum of the ranks of the blocks, that is,
\begin{equation}
R_d = r(1,d) + r(2,d) + \cdots + r(d,d).
\end{equation}

Consider the $j$th block $B_j$.  To compute the rank of
$B_j$, we will perform column reduction on the block.
Recall that $\delta \not= 0$.
It is sufficient to consider as pivots the entries in the matrix that
correspond to the negative integers ($-b_2$) and the $b_1 \delta$,
at least when $b_1$ and $b_2$ are nonzero.
And for the column reduction, we only need to take into account that these
entries are nonzero.
For example, the block $B_4$ for $d=9$
can be written and further subdivided as follows:

{
\tiny
\begin{equation*}
\arraycolsep=1pt
\begin{array}{c||ccccc|ccccc|ccccc|ccccc|ccccc|ccccc}
 & \mathbin{\rotatebox[origin=c]{90}{$z_1^{5}\bar{z}_1^{4}$}} & \mathbin{\rotatebox[origin=c]{90}{$z_1^{5}\bar{z}_1^{3}\bar{z}_2$}} & \mathbin{\rotatebox[origin=c]{90}{$z_1^{5}\bar{z}_1^{2}\bar{z}_2^{2}$}} & \mathbin{\rotatebox[origin=c]{90}{$z_1^{5}\bar{z}_1\bar{z}_2^{3}$}} & \mathbin{\rotatebox[origin=c]{90}{$z_1^{5}\bar{z}_2^{4}$}} & \mathbin{\rotatebox[origin=c]{90}{$z_1^{4}z_2\bar{z}_1^{4}$}} & \mathbin{\rotatebox[origin=c]{90}{$z_1^{4}z_2\bar{z}_1^{3}\bar{z}_2$}} & \mathbin{\rotatebox[origin=c]{90}{$z_1^{4}z_2\bar{z}_1^{2}\bar{z}_2^{2}$}} & \mathbin{\rotatebox[origin=c]{90}{$z_1^{4}z_2\bar{z}_1\bar{z}_2^{3}$}} & \mathbin{\rotatebox[origin=c]{90}{$z_1^{4}z_2\bar{z}_2^{4}$}} & \mathbin{\rotatebox[origin=c]{90}{$z_1^{3}z_2^{2}\bar{z}_1^{4}$}} & \mathbin{\rotatebox[origin=c]{90}{$z_1^{3}z_2^{2}\bar{z}_1^{3}\bar{z}_2$}} & \mathbin{\rotatebox[origin=c]{90}{$z_1^{3}z_2^{2}\bar{z}_1^{2}\bar{z}_2^{2}$}} & \mathbin{\rotatebox[origin=c]{90}{$z_1^{3}z_2^{2}\bar{z}_1\bar{z}_2^{3}$}} & \mathbin{\rotatebox[origin=c]{90}{$z_1^{3}z_2^{2}\bar{z}_2^{4}$}} & \mathbin{\rotatebox[origin=c]{90}{$z_1^{2}z_2^{3}\bar{z}_1^{4}$}} & \mathbin{\rotatebox[origin=c]{90}{$z_1^{2}z_2^{3}\bar{z}_1^{3}\bar{z}_2$}} & \mathbin{\rotatebox[origin=c]{90}{$z_1^{2}z_2^{3}\bar{z}_1^{2}\bar{z}_2^{2}$}} & \mathbin{\rotatebox[origin=c]{90}{$z_1^{2}z_2^{3}\bar{z}_1\bar{z}_2^{3}$}} & \mathbin{\rotatebox[origin=c]{90}{$z_1^{2}z_2^{3}\bar{z}_2^{4}$}} & \mathbin{\rotatebox[origin=c]{90}{$z_1z_2^{4}\bar{z}_1^{4}$}} & \mathbin{\rotatebox[origin=c]{90}{$z_1z_2^{4}\bar{z}_1^{3}\bar{z}_2$}} & \mathbin{\rotatebox[origin=c]{90}{$z_1z_2^{4}\bar{z}_1^{2}\bar{z}_2^{2}$}} & \mathbin{\rotatebox[origin=c]{90}{$z_1z_2^{4}\bar{z}_1\bar{z}_2^{3}$}} & \mathbin{\rotatebox[origin=c]{90}{$z_1z_2^{4}\bar{z}_2^{4}$}} & \mathbin{\rotatebox[origin=c]{90}{$z_2^{5}\bar{z}_1^{4}$}} & \mathbin{\rotatebox[origin=c]{90}{$z_2^{5}\bar{z}_1^{3}\bar{z}_2$}} & \mathbin{\rotatebox[origin=c]{90}{$z_2^{5}\bar{z}_1^{2}\bar{z}_2^{2}$}} & \mathbin{\rotatebox[origin=c]{90}{$z_2^{5}\bar{z}_1\bar{z}_2^{3}$}} & \mathbin{\rotatebox[origin=c]{90}{$z_2^{5}\bar{z}_2^{4}$}}\\
\hline
\hline
z_1^{6}\bar{z}_1^{3} & . & \boxed{*} & . & . & . & . & . & . & . & . & . & . & . & . & . & . & . & . & . & . & . & . & . & . & . & . & . & . & . & .\\
z_1^{6}\bar{z}_1^{2}\bar{z}_2 & . & . & \boxed{*} & . & . & . & . & . & . & . & . & . & . & . & . & . & . & . & . & . & . & . & . & . & . & . & . & . & . & .\\
z_1^{6}\bar{z}_1\bar{z}_2^{2} & . & . & . & \boxed{*} & . & . & . & . & . & . & . & . & . & . & . & . & . & . & . & . & . & . & . & . & . & . & . & . & . & .\\
z_1^{6}\bar{z}_2^{3} & . & . & . & . & \boxed{*} & . & . & . & . & . & . & . & . & . & . & . & . & . & . & . & . & . & . & . & . & . & . & . & . & .\\
\hline
z_1^{5}z_2\bar{z}_1^{3} & \boxed{*} & ? & . & . & . & . & * & . & . & . & . & . & . & . & . & . & . & . & . & . & . & . & . & . & . & . & . & . & . & .\\
z_1^{5}z_2\bar{z}_1^{2}\bar{z}_2 & . & * & ? & . & . & . & . & \boxed{*} & . & . & . & . & . & . & . & . & . & . & . & . & . & . & . & . & . & . & . & . & . & .\\
z_1^{5}z_2\bar{z}_1\bar{z}_2^{2} & . & . & * & ? & . & . & . & . & \boxed{*} & . & . & . & . & . & . & . & . & . & . & . & . & . & . & . & . & . & . & . & . & .\\
z_1^{5}z_2\bar{z}_2^{3} & . & . & . & * & ? & . & . & . & . & \boxed{*} & . & . & . & . & . & . & . & . & . & . & . & . & . & . & . & . & . & . & . & .\\
\hline
z_1^{4}z_2^{2}\bar{z}_1^{3} & . & . & . & . & . & \boxed{*} & ? & . & . & . & . & * & . & . & . & . & . & . & . & . & . & . & . & . & . & . & . & . & . & .\\
z_1^{4}z_2^{2}\bar{z}_1^{2}\bar{z}_2 & . & . & . & . & . & . & \boxed{*} & ? & . & . & . & . & * & . & . & . & . & . & . & . & . & . & . & . & . & . & . & . & . & .\\
z_1^{4}z_2^{2}\bar{z}_1\bar{z}_2^{2} & . & . & . & . & . & . & . & * & ? & . & . & . & . & \boxed{*} & . & . & . & . & . & . & . & . & . & . & . & . & . & . & . & .\\
z_1^{4}z_2^{2}\bar{z}_2^{3} & . & . & . & . & . & . & . & . & * & ? & . & . & . & . & \boxed{*} & . & . & . & . & . & . & . & . & . & . & . & . & . & . & .\\
\hline
z_1^{3}z_2^{3}\bar{z}_1^{3} & . & . & . & . & . & . & . & . & . & . & \boxed{*} & ? & . & . & . & . & * & . & . & . & . & . & . & . & . & . & . & . & . & .\\
z_1^{3}z_2^{3}\bar{z}_1^{2}\bar{z}_2 & . & . & . & . & . & . & . & . & . & . & . & \boxed{*} & ? & . & . & . & . & * & . & . & . & . & . & . & . & . & . & . & . & .\\
z_1^{3}z_2^{3}\bar{z}_1\bar{z}_2^{2} & . & . & . & . & . & . & . & . & . & . & . & . & \boxed{*} & ? & . & . & . & . & * & . & . & . & . & . & . & . & . & . & . & .\\
z_1^{3}z_2^{3}\bar{z}_2^{3} & . & . & . & . & . & . & . & . & . & . & . & . & . & * & ? & . & . & . & . & \boxed{*} & . & . & . & . & . & . & . & . & . & .\\
\hline
z_1^{2}z_2^{4}\bar{z}_1^{3} & . & . & . & . & . & . & . & . & . & . & . & . & . & . & . & \boxed{*} & ? & . & . & . & . & * & . & . & . & . & . & . & . & .\\
z_1^{2}z_2^{4}\bar{z}_1^{2}\bar{z}_2 & . & . & . & . & . & . & . & . & . & . & . & . & . & . & . & . & \boxed{*} & ? & . & . & . & . & * & . & . & . & . & . & . & .\\
z_1^{2}z_2^{4}\bar{z}_1\bar{z}_2^{2} & . & . & . & . & . & . & . & . & . & . & . & . & . & . & . & . & . & \boxed{*} & ? & . & . & . & . & * & . & . & . & . & . & .\\
z_1^{2}z_2^{4}\bar{z}_2^{3} & . & . & . & . & . & . & . & . & . & . & . & . & . & . & . & . & . & . & \boxed{*} & ? & . & . & . & . & * & . & . & . & . & .\\
\hline
z_1z_2^{5}\bar{z}_1^{3} & . & . & . & . & . & . & . & . & . & . & . & . & . & . & . & . & . & . & . & . & \boxed{*} & ? & . & . & . & . & * & . & . & .\\
z_1z_2^{5}\bar{z}_1^{2}\bar{z}_2 & . & . & . & . & . & . & . & . & . & . & . & . & . & . & . & . & . & . & . & . & . & \boxed{*} & ? & . & . & . & . & * & . & .\\
z_1z_2^{5}\bar{z}_1\bar{z}_2^{2} & . & . & . & . & . & . & . & . & . & . & . & . & . & . & . & . & . & . & . & . & . & . & \boxed{*} & ? & . & . & . & . & * & .\\
z_1z_2^{5}\bar{z}_2^{3} & . & . & . & . & . & . & . & . & . & . & . & . & . & . & . & . & . & . & . & . & . & . & . & \boxed{*} & ? & . & . & . & . & *\\
\hline
z_2^{6}\bar{z}_1^{3} & . & . & . & . & . & . & . & . & . & . & . & . & . & . & . & . & . & . & . & . & . & . & . & . & . & \boxed{*} & ? & . & . & .\\
z_2^{6}\bar{z}_1^{2}\bar{z}_2 & . & . & . & . & . & . & . & . & . & . & . & . & . & . & . & . & . & . & . & . & . & . & . & . & . & . & \boxed{*} & ? & . & .\\
z_2^{6}\bar{z}_1\bar{z}_2^{2} & . & . & . & . & . & . & . & . & . & . & . & . & . & . & . & . & . & . & . & . & . & . & . & . & . & . & . & \boxed{*} & ? & .\\
z_2^{6}\bar{z}_2^{3} & . & . & . & . & . & . & . & . & . & . & . & . & . & . & . & . & . & . & . & . & . & . & . & . & . & . & . & . & \boxed{*} & ?\\
\end{array}

\end{equation*}
}

Nonzero entries (the negative integers and the multiples of $\delta$)
are marked by ``$*$'',
entries that are
possibly zero (the multiples of $\beta$) by ``$?$'',
and entries that are zero by dots.
The boxed stars are the pivots for column reduction---more on this below.
The $j$th block $B_j$ has sub-blocks defined as follows:
For each fixed monomial $z^\alpha$, for $\abs{\alpha} = d-j < d$,
take the columns corresponding to
the monomials $z^{\alpha} \bar{z}^{\gamma}$ where $\abs{\gamma} = j$.
There are $d-j+1$ such sub-blocks because that is the number of monomials of
the form $z^\alpha$ with $\abs{\alpha} = d-j$.
These sub-blocks are no longer in the form of a direct sum---they overlap.  We divide these blocks into a top half and a
bottom half, depending on the holomorphic part of the output monomial.

We wish to show that $B_j$ is of full rank.
To see this fact, we show that each column has one nonzero entry
that can be used as a pivot.
When we say ``nonzero entry'' below we mean one of the
starred entries, that is, either the negative integer or the entry with
$\delta$, and by ``possibly nonzero entry'' we mean the question marks,
that is, the multiples of $\beta$.

In the first sub-block we consider all the nonzero
entries in the top half of the sub-block and the nonzero entry in the
first column of the bottom half of the sub-block.

By a column operation,
the pivot in the bottom half of the first sub-block can be used
zero out
the first nonzero entry in the top half of the second sub-block.
The first column of the second sub-block can be used to zero out
the possibly nonzero entry in the second column.
We therefore use as a pivot the nonzero entry in the second column of the
bottom half of the second sub-block.  In the remaining columns of the second
sub-block we use the nonzero entries in the top half as pivots.

The two pivots in the bottom half of
the second sub-block can zero out the first two nonzero entries in the top
half of the third sub-block.  Similarly, the possibly nonzero entries can
be zeroed out in the second and third column of the third sub-block, and the
nonzero entries in the first three columns in the bottom half of
the third sub-block are pivots.

In each further sub-block, we can zero out an additional nonzero entry in
the top half.  We continue this procedure until all the entries in the top
half of the sub-block can be zeroed out.  In all further sub-blocks
we use the nonzero entries in the bottom half as pivots.

All in all, we find that as we move through the sub-blocks from left to right,
we find pivots in every column until we run out of nonzero entries in the
top half of the sub-block.
If the matrix $B_j$ does not have enough sub-blocks so that
we never run out of the nonzero entries in the top half, then
$B_j$ has a pivot entry in every column.
If on the other hand $B_j$ has more sub-blocks,
then in the blocks after we ran out of nonzero entries in
the top half,  every nonzero
entry in the bottom half of the sub-block is a pivot.
Therefore we have a pivot in
every row of $B_j$.  In $B_j$, there are $d-j+1$ sub-blocks,
and in each sub-block there are $j$ nonzero entries in the top half.
As there is a pivot in every column if $d-j+1 \leq j$
or in every row if $d-j+1 > j$, we find that $B_j$ is of full rank.
In other words, the rank of $B_j$ is
\begin{equation}
r(j,d) =
\begin{cases}
(j+1)(d-j+1) & \text{if } d-j+1 \leq j , \\
j(d-j+2) & \text{if } d-j+1 > j .
\end{cases}
\end{equation}

Thus, the rank of $X_d$ is
\begin{equation}
R_d= \sum_{j=\lceil\frac{d+1}{2}\rceil}^d (j+1)(d-j+1)+\sum_{j=1}^{\lceil\frac{d+1}{2}\rceil-1} j(d-j+2)   .
\end{equation}
In order to partially combine these sums so that the first one starts at $j=1$, we rewrite the second summand as $(j+1)(d-j+1)+(2j-d-1)$. Using $\lceil\frac{d+1}{2}\rceil-1=\lfloor\frac{d}{2}\rfloor$, we get
\begin{equation} \label{eq:sumr}
R_d= \sum_{j=1}^d (j+1)(d-j+1)+\sum_{j=1}^{\lfloor\frac{d}{2}\rfloor} (2j-d-1).
\end{equation}
Thus, to prove the desired formula \eqref{eq:rank} for $R_d$ we need to show that the second sum in \eqref{eq:sumr} equals 
\begin{equation}\sum_{j=1}^d \left\{2\left\lfloor \frac{j+1}{2} \right\rfloor-(j+1)\right\} (d-j+1).\end{equation}
The term in braces equals $0$ when $j$ is odd and equals $-1$ when $j$ is even. We rewrite the sum as 
\begin{equation}-\sum_{k=1}^{\lfloor\frac{d}{2}\rfloor}  (d-2k+1).\end{equation} Clearly this equals the second sum in \eqref{eq:sumr}. 
The claim is proved.
\end{proof}

Let ${\mathit{CR}}^d(M)$ be the space of degree-$d$ homogeneous
polynomials
$f(z,\bar{z})$ that, when considered as
functions on $M$ (parametrized by $z$), are CR functions on
$M_{\mathit{CR}}$. Then ${\mathit{CR}}^d(M)$ has dimension $\binom{d+3}{3}-R_d$. It is not hard to show that
\begin{equation}
\binom{d+3}{3} - \sum_{j=1}^d 2\left\lfloor \frac{j+1}{2} \right\rfloor (d-j+1)  
=
\left\lfloor
\frac{{(d+2)}^2}{4}
\right\rfloor ,
\end{equation} so by our formula for $R_d$ the dimension of ${\mathit{CR}}^d(M)$ is $\left\lfloor
\frac{{(d+2)}^2}{4}
\right\rfloor $.
But the dimension of the space of weighted homogeneous polynomials
 in $z$ and $w$ of degree $d$ is
\begin{equation}
\sum\limits_{j+2k=d} (j+1)
=
\sum_{k=0}^{\lfloor d/2 \rfloor}
(d-2k+1)
=
\left\lfloor
\frac{{(d+2)}^2}{4}
\right\rfloor .
\end{equation}
Because the dimension of ${\mathit{CR}}^d(M)$ is the same, the restrictions
to $M$ of the weighted homogeneous polynomials in $z$ and $w = Q(z,\bar{z})$
of degree $d$
span ${\mathit{CR}}^d(M)$.

If $f(z,\bar{z})$ is any polynomial that is CR on $M_{{\mathit{CR}}}$, then equation \eqref{eq:homogL}
implies that the homogeneous parts of $f$ are CR, and hence are in the span
of monomials in $z$ and $w = Q(z,\bar{z})$.
\end{proof}

\begin{lemma} \label{lem:PolyExtnGeneraln}
In the setup of Theorem \ref{thm:PolyExtn},
(a) implies (b).
\end{lemma}

\begin{proof}
By Lemmas \ref{lem:PolyExtnDE} and \ref{lem:PolyExtnMa},
(a) implies (b) when $n=2$.
It suffices to prove
that (a) implies (b) for $n > 2$.

Choose a linear map $R \colon \C^2 \to \C^n$, and define
$M_R \subset \C^2 \times \C$ by
\begin{equation}
w = Q(R\xi,\overline{R\xi})
\end{equation}
for variables $(\xi,w) \in \C^2 \times \C$.  The rank condition on $M$
guarantees that $M_R$ satisfies the rank condition as well for an open dense
set of $R$.  Therefore we have that (a) implies (b) for $M_R$ for an open
dense set of $R$.

Consider $R$ of the form
\begin{equation}
\begin{bmatrix}
1&0\\
0&1\\
\omega_1 & \omega_2
\end{bmatrix}
\end{equation}
for some column vectors $\omega_j\in\C^{n-2}$.

Let $f(z,\bar{z})$ be a polynomial homogeneous of degree $d$ that, when considered as a
   function on $M$, is a CR function. Because (a) implies (b) for $M_R$, 
we find a polynomial $F_R(\xi,w)$ such that
\begin{equation} \label{eqn:polyextnslices}
f(R\xi,\overline{R\xi})
= F_R\bigl(\xi,Q(R\xi,\overline{R\xi})\bigr)
= \sum\limits_{2j+\abs{\alpha}=d}\,
c_{\alpha j}(\omega_1,\omega_2,\bar{\omega}_1,\bar{\omega}_2)
\xi^\alpha Q^j .
\end{equation}
We first need to prove that all the coefficients $c_{\alpha j}$ are independent
of $\bar{\omega}_1$ and $\bar{\omega}_2$.
The function $f(R\xi,\overline{R\xi})$ is a CR function on the
quadric in $\xi,\omega,w$ coordinates $w = Q(R\xi,\overline{R\xi}) =
\widetilde{Q}(\xi,\omega,\bar{\xi},\bar{\omega})$.
For simplicity let us assume $n=3$, as the argument is the same for
higher dimensions.  That is, let us suppose that $\omega_j \in \C$.

The quadric $w = \widetilde{Q}$ given above
can be written as
\begin{equation}
w=Q(\xi_1,\xi_2,\omega_1 \xi_1 +
\omega_2\xi_2,\bar{\xi}_1,\bar{\xi}_2,\overline{\omega_1 \xi_1 + \omega_2\xi_2})
=
\widetilde{Q}(\xi_1,\xi_2,\omega_1,\omega_2,\bar{\xi}_1,\bar{\xi}_2,\bar{\omega_1},\bar{\omega}_2) .
\end{equation}
For $\ell=1,2$ consider the CR vector field
\begin{equation}
L = L_\ell
=
\widetilde{Q}_{\bar{\xi}_\ell}\frac{\partial}{\partial\bar{\omega}_\ell}
-
\widetilde{Q}_{\bar{\omega}_\ell}\frac{\partial}{\partial\bar{\xi}_\ell}
=
\left(Q_{\bar{z}_\ell}+Q_{\bar{z}_3}\bar{\omega}_\ell\right)\frac{\partial}{\partial\bar{\omega}_\ell}
-
Q_{\bar{z}_3}\bar{\xi}_\ell\frac{\partial}{\partial\bar{\xi}_\ell}.
\end{equation}
Since $f$ is CR, applying $L$ to \eqref{eqn:polyextnslices} gives
us that the $c_{\alpha j}$'s are independent of $\bar{\omega}_1$ and $\bar{\omega}_2$.
That is because $\xi^\alpha Q^j$ are CR and hence $L$ only hits the $c_{\alpha j}$.

Now suppose $\omega_1 = 0$.  We have
\begin{equation}
f(
\xi_1,\xi_2, \omega_2 \xi_2,
\bar{\xi}_1,\bar{\xi}_2, \bar{\omega}_2 \bar{\xi}_2
)
= \sum\limits_{2j+\abs{\alpha}=d}\, c_{\alpha
j}(0,\omega_2)\xi^\alpha Q^j .
\end{equation}
Setting $\xi_1 = z_1$,
$\xi_2 = z_2$,
$\omega_2 = \frac{z_3}{z_2}$, we obtain
\begin{equation}
f(z,\bar{z})
= \sum\limits_{2j+\abs{\alpha}=d}\, c_{\alpha
j}\biggl(0,\frac{z_3}{z_2}\biggr)z_1^{\alpha_1}z_2^{\alpha_2} Q^j .
\end{equation}
We have a rational extension to $\C^{4}$,
with a possible pole when $z_2 = 0$.  The same argument with $\omega_2
= 0$ obtains another rational extension with a possible pole at $z_1 = 0$.
Outside any possible poles the extensions are identical as the holomorphic
extension near CR points is unique.  The poles are therefore
only on the set $z_2 = z_1 = 0$, which implies that there are no poles.
Thus we find a polynomial extension $F(z,w)$.
The argument for $n > 3$ follows in the same way.
\end{proof}

\begin{proof}[Proof of Theorem~\ref{thm:PolyExtn}]
Proposition~\ref{prop:LinExtnFails2} shows that 
(c) $\Rightarrow$ (a).  
Next suppose that (b) is true.  If $h(z,\bar{z})$
is (real) linear and CR, then it is
the restriction of a holomorphic function in $z$ only since $w$ is of weight
two.  The conclusion of (c) follows because $M$ is parametrized by $z$.  
Lemma \ref{lem:PolyExtnGeneraln} proves that (a) $\Rightarrow$ (b).
That (a) $\Leftrightarrow$ (d) follows from
Proposition~\ref{prop:equivofAandD}.
\end{proof}

%%%%%%%%%%%%%%%%%%%%%%%%%%%%%%%%%%%%%%%%%%%%%%%%%%%%%%%%%%%%%%%%%%%%%%%%%%%%

\section{Extending real-analytic CR functions}\label{sec:real-analytic}

In this section we use Theorem~\ref{thm:PolyExtn} to prove 
Theorem~\ref{thm:mainext}. First we obtain a formal extension.

\begin{lemma} \label{lem:formalext}
Let $(z,w) \in \C^n \times \C$ be the coordinates and
near the origin, let $M \subset \C^{n+1}$ be a codimension-2 submanifold given by
\begin{equation}
w = \rho(z,\bar{z}) = Q(z,\bar{z}) + E(z,\bar{z}) ,
\end{equation}
where $\rho$ is real-analytic and $O(\snorm{z}^2)$ and $E$ is
$O(\snorm{z}^3)$.
Assume $\bar{\partial}Q\not\equiv 0$.
Let $M^{\mathit{quad}}$ be the quadric model defined by $w = Q(z,\bar{z})$,
and suppose that $M^{\mathit{quad}}$ satisfies one of the conditions of
Theorem~\ref{thm:PolyExtn}.

Suppose $f(z,\bar{z})$ is a real-analytic function defined near the origin that, when considered as a
function on $M$ (parametrized by $z$), is a CR function on $M_{\mathit{CR}}$.

Then there exists a unique formal power series $F(z,w)$ such that $f$ and $F$
agree on $M$, that is,
      \begin{equation}
      f(z,\bar{z}) = F\bigl(z, \rho(z,\bar{z}) \bigr),
\qquad \text{or} \qquad
      f(z,\bar{z}) = F(z,w) + a(z,\bar{z},w) \bigl(w - \rho(z,\bar{z}) \bigr) ,
      \avoidbreak
      \end{equation}
\enlargethispage{\baselineskip}
as formal power series (for some formal power series $a$).
\end{lemma}

\begin{proof}
Suppose the order of $f$ at the origin is $k$ and write
\begin{equation}
f(z,\bar{z}) = f_k(z,\bar{z}) + \tilde{f}(z,\bar{z}) ,
\end{equation}
where $f_k$ is the degree-$k$ homogeneous part of $f$.
CR vector fields on $M$ have the form
\begin{equation}
L = L_{j,\ell}
=
\rho_{\bar{z}_\ell} \frac{\partial}{\partial \bar{z}_j}
-
\rho_{\bar{z}_j} \frac{\partial}{\partial \bar{z}_\ell}
=
\bigl(Q_{\bar{z}_\ell} +  E_{\bar{z}_\ell} \bigr) \frac{\partial}{\partial \bar{z}_j}
-
\bigl(Q_{\bar{z}_j} +  E_{\bar{z}_j} \bigr) \frac{\partial}{\partial \bar{z}_\ell} .
\end{equation}
CR vector fields on $M^{\mathit{quad}}$ 
have the form
\begin{equation}
L^{\mathit{quad}} =
L^{\mathit{quad}}_{j,\ell} =
Q_{\bar{z}_\ell} \frac{\partial}{\partial \bar{z}_j}
-
Q_{\bar{z}_j} \frac{\partial}{\partial \bar{z}_\ell} .
\end{equation}
Then
\begin{equation}
0 = Lf = Lf_k + L\tilde{f}
=
Q_{\bar{z}_\ell} \frac{\partial f_k}{\partial \bar{z}_j}
-
Q_{\bar{z}_j} \frac{\partial f_k}{\partial \bar{z}_\ell} 
+ O(\snorm{z}^{k+1})
=
L^{\mathit{quad}} f_k + O(\snorm{z}^{k+1}) .
\end{equation}
As $L^{\mathit{quad}} f_k$ is of order $k$, we have that $f_k$ is CR on
$M^{\mathit{quad}}$.  Hence by Theorem~\ref{thm:PolyExtn}
there exists a weighted homogeneous $F_k(z,w)$
such that
\begin{equation}
f_k(z,\bar{z}) =
F_k\bigl(z,Q(z,\bar{z})\bigr) .
\end{equation}
The function
$g(z,\bar{z}) = F_k\bigl(z,\rho(z,\bar{z})\bigr) =
F_k\bigl(z,Q(z,\bar{z})+E(z,\bar{z})\bigr)$ is a CR function on $M$,
and furthermore the $k$th order part of $g$ is equal to $f_k$.
Now consider the function $h=f-g$.  The function $h$
is CR, and is of order at least $k+1$.
By induction therefore we obtain a formal power series $F$.  The series
$F$ is unique as the $F_k$ in each step giving the $k$th order terms is unique.
\end{proof}

\begin{lemma}\label{lem:ConfFormalPowerSeries}
Let $M \subset \C^{2}$ be a
real-analytic submanifold of real codimension 2 with a
CR singularity
at $0 \in M$. Assume that $M$ is defined by $w = \rho(z,\bar{z})$, for $(z,w) \in \C^2$,
where $\rho$ is  $O(\snorm{z}^2)$,
and assume that the power series of $\rho(z,\bar{z})$ at $0$ contains a nonzero term of the form
$\bar{z}^k$ or
$z \bar{z}$.

Suppose $f$ is a real-analytic function on $M$ that admits a formal power series
$F(z,w)$, that is, formally for some formal power series $a$,
%$F\bigl(z,\rho(z,\bar{z})\bigr)$ is equal to $f$ formally.
      \begin{equation}
      f(z,\bar{z}) = F\bigl(z, \rho(z,\bar{z}) \bigr),
\qquad \text{or} \qquad
      f(z,\bar{z}) = F(z,w) + a(z,\bar{z},w) \bigl(w - \rho(z,\bar{z}) \bigr) .
      \qquad
      \end{equation}
Then $F$ is convergent.
\end{lemma}

The proof is for the most part contained in \cite{crext2}, but since the
proof is not long we prove it again in the full generality needed in this
paper.

\begin{proof}
Parametrizing $M$ by $z$, we write $f(z,\bar{z})$ for the value of $f$ on $M$ at
$\bigl(z,\rho(z,\bar{z})\bigr)$ as usual.
We may locally complexify and treat $z$ and $\bar{z}$ as independent
variables.

\textbf{Case 1:}
The power series of $\rho(z,\bar{z})$ contains $\bar{z}^k$.

Suppose that $k\geq 2$ is the smallest $k$ for which a term $\bar{z}^k$ exists.
The equation for $M$ is of the form
\begin{equation}
   c \bar{z}^k + E(z,\bar{z}) - w = 0 ,
\end{equation}
for some $c \not= 0$. We complexify and consider the equation as
an equation in $z,\bar{z},w$ (treating $\bar{z}$ as independent variable).
Using the Weierstrass preparation theorem,
we may locally solve for $\bar{z}$ in terms of $w$ and
$z$.  Let us denote these solutions by $\xi_1(z,w)$, \ldots, $\xi_k(z,w)$;
if $(z,w) \in M$, then
one of these is the complex conjugate of $z$.
The $\xi_j$'s are not holomorphic, but any holomorphic symmetric function of
$\xi_1,\ldots,\xi_k$ is holomorphic.
(See, e.g., \cite{Whitney:book}*{Lemma 8A in Chapter 1}.)
So,
\begin{equation} \label{eq:avgout}
\widetilde{F}(z,w) = \frac{1}{k}\sum\limits_{j=1}^k f\bigl(z,\xi_j(z,w)\bigr)
\end{equation}
is a well-defined holomorphic function in a neighborhood of the origin.

Given variables $\zeta_1,\ldots,\zeta_k$, consider as a formal power series
\begin{equation}
\begin{split}
\frac{1}{k} \sum_{j=1}^k f(z,\zeta_j)
& =
\frac{1}{k} \sum_{j=1}^k \Bigl(
F(z,w) + a(z,\zeta_j,w) \bigl(w - \rho(z,\zeta_j) \bigr)
\Bigr)
\\
& =
F(z,w) + 
\frac{1}{k} \sum_{j=1}^k
a(z,\zeta_j,w) \bigl(w - \rho(z,\zeta_j) \bigr) .
\end{split}
\end{equation}

We wish to show that the second term on the right is zero when we plug in
$\xi_1,\ldots,\xi_k$.  Of course this is a formal power series, so we
cannot just plug in.  It is, however, symmetric in 
$\zeta_1,\ldots,\zeta_k$,
 so it is a formal power series in the elementary symmetric polynomials of
$\zeta_1,\ldots,\zeta_k$, and the elementary symmetric polynomials of
$\xi_1,\ldots,\xi_k$ are well-defined power series.
For some large $m$, consider $a_m$ to be the terms of order at most $m$.
Then formally,
\begin{equation}
\frac{1}{k} \sum_{j=1}^k
a_m(z,\zeta_j,w) \bigl(w - \rho(z,\zeta_j) \bigr)
=
\frac{1}{k} \sum_{j=1}^k
a(z,\zeta_j,w) \bigl(w - \rho(z,\zeta_j) \bigr)
\quad + \quad
O\left( {\snorm{(z,w)}}^{m+1}\right) .
\end{equation}
The left-hand side is a well-defined function symmetric
in $\zeta_1,\ldots,\zeta_k$ and is zero when we plug in
$\xi_1,\ldots,\xi_k$.  The symmetric functions of $\xi_j$
are of order at least 1, and the elementary symmetric polynomials have degree at most $k$.
Consequently,
\begin{equation}
\frac{1}{k} \sum_{j=1}^k
a(z,\xi_j,w) \bigl(w - \rho(z,\xi_j) \bigr)
=
0
+
O\left( {\snorm{(z,w)}}^{
\left\lfloor \frac{m+1}{k} \right\rfloor
}\right) .
\end{equation}
That is, as a formal power series,
$\widetilde{F}(z,w) = F(z,w)$.  Because $\widetilde{F}(z,w)$ converges,  so
does $F$.

\textbf{Case 2.} The power series of $\rho(z,\bar{z})$ contains $z \bar{z}$,
but no term of the form $\bar{z}^k$.

Via Moser~\cite{Moser85}, after a local 
biholomorphic change of variables at the origin, $M$ is given by
\begin{equation}
w = z\bar{z} .
\end{equation}
The power series
\begin{equation}
f(z,\bar{z}) =
\sum_{j,k} c_{k,j} z^k {\bar{z}}^j ,
\end{equation}
can be written in terms of $z$ and $w = z \bar{z}$
as
\begin{equation}
F(z,w) = \sum_{j,\ell} d_{\ell,j} z^\ell w^j .
\end{equation}
In particular, $c_{k,j} = 0$ if $k < j$, and
$d_{\ell,j} = c_{\ell+j,j}$.
The Cauchy estimates give
$\sabs{c_{k,j}} \leq \frac{M}{\epsilon^{k+j}}$ for some $\epsilon > 0$.
Because
$\sabs{d_{\ell,j}} = \sabs{c_{\ell+j,j}} \leq \frac{M}{\epsilon^{\ell}
\epsilon^{2j}}$, the series $F$ converges.
\end{proof}

\begin{proof}[Proof of Theorem~\ref{thm:mainext}]
It follows from the condition
$\rank \left[ \begin{smallmatrix} A^* \\ B \end{smallmatrix} \right] \geq 2$
that $\bar{\partial} Q\not\equiv 0$.
Apply Lemma~\ref{lem:formalext} to find a formal solution $F(z,w)$
that is equal to $f$ on $M$ (formally).

Consider a nonzero $c \in \C^n$.  Using coordinates
$(\xi,w) \in \C \times \C$, let $M_c$ be given by
\begin{equation}
w = \rho(c\xi,\overline{c\xi}) .
\end{equation}
The function $f(c\xi,\overline{c\xi})$ is equal (formally on $M$) to the
formal power series $F(c\xi,w)$.

Because $\bar{\partial} Q\not\equiv 0$, for an open dense set of $c \in \C^n$,
$\rho(c\xi,\overline{c\xi})$
has a nonzero nonholomorphic quadratic term; then Lemma~\ref{lem:ConfFormalPowerSeries}
applies, and $F(c\xi,w)$ converges.
Therefore, $F(z,w)$ converges
via a standard Baire category argument
(see e.g.\ \cite{BER:book}*{Theorem 5.5.30}).
\end{proof}

%%%%%%%%%%%%%%%%%%%%%%%%%%%%%%%%%%%%%%%%%%%%%%%%%%%%%%%%%%%%%%%%%%%%%%%%%%%%

\section{Flattening} \label{sec:flattening}

A well-known problem for CR singular manifolds is to determine when
a CR singular manifold is flattenable, that is, a subset of a Levi-flat hypersurface.
In other words, $M$ is flattenable if there is locally
a nonconstant holomorphic function with
nonvanishing derivative on $M$ that is real-valued on $M$.
For codimension-2 nondegenerate (in the sense that the matrix $A$
is nondegenerate)
CR singular manifolds in 3 or more
dimensions, the problem has been almost completely solved by
Fang--Huang~\cite{FangHuang}, where they prove that the necessary condition of
nowhere minimality is sufficient in all but one unresolved case in  $\C^3$.
The question still remains in that unresolved exceptional case and for
degenerate manifolds.

As an application of our extension result, we give a new way of checking
flattenability that we hope will yield a complete solution.  The condition
we propose is an existence of a first integral of the complex tangent
bundle.  By this we mean a function $g$ defined on $M$ that is constant
on integral curves of $T^cM = TM \cap J(TM)$, where $J$ is the complex
structure.  In other words, $g$ is a function that is constant on the CR
orbits of $M_{\mathit{CR}}$.

\begin{cor}
Let $(z,w) \in \C^{n} \times \C$ be the coordinates and,
near the origin, let $M \subset \C^{n+1}$ be a codimension-2 submanifold given by
\begin{equation}
w = \rho(z,\bar{z}) = Q(z,\bar{z}) + E(z,\bar{z}) =
z^*Az + \overline{z^tBz} + z^tCz +
E(z, \bar{z}),
\end{equation}
where $\rho$ is real-analytic,
$A,B,C$ are complex $n \times n$  matrices, $B$ and $C$ are symmetric,
and $E$ is $O(\snorm{z}^3)$.
Further suppose that 
\begin{equation}
\rank \begin{bmatrix} A^* \\ B \end{bmatrix} \geq 2 .
\end{equation}

Then the following conditions are equivalent:
\begin{enumerate}[(a)]
  \item
$M$ is flattenable, that is, there exists a holomorphic function near the
origin with nonvanishing derivative that is real-valued on $M$.
  \item
The complex tangent bundle of $M$ has a real-valued nonconstant real-analytic first
integral $g$ defined on $M$
whose second derivatives in terms of $z$ and $\bar{z}$ do not all vanish
at the origin.
\end{enumerate}
\end{cor}

A necessary condition for $M$ to be flattenable is that $A$
can be made real-valued after a linear change of coordinates
(see~\cite{DTZ}).
After the normalization making $C = B$, we can therefore also assume that
$Q$ is real-valued.  In this case, if $g$ exists, then
$g(z,\bar{z}) = \alpha Q(z,\bar{z}) + O(\snorm{z}^3)$ at the origin
for some nonzero real $\alpha$.

\begin{proof}
(a) $\Rightarrow$ (b) follows by restricting the flattening
holomorphic function $f$ to $M$.  Suppose that $Q$ is normalized to
be real-valued as we mentioned above.
It is not difficult to see that $f$,
when expanded at the origin,
cannot have linear terms in $z$ and that it must therefore have a linear
term in $w$ with coefficient $\alpha$.
Because $w = \rho(z,\bar{z})$, the function
$g(z,\bar{z}) = f\bigl(z,\rho(z,\bar{z})\bigr)
= \alpha Q(z,\bar{z}) + O(\snorm{z}^3)$ is the
required first integral.

To show (b) $\Rightarrow$ (a), suppose that $g$ is the given first integral.
Because $g$ is constant on the CR orbits of $M_{\mathit{CR}}$, it is a CR function.  It
therefore extends to a holomorphic function $f(z,w)$ such that $g(z,\bar{z})
= f\bigl(z,\rho(z,\bar{z})\bigr)$.
There are two possibilities for the
quadratic terms of $g$: either quadratic holomorphic terms in
$z$, or constant multiples of $Q$.  The quadratic terms must be
real-valued, and therefore there can be no holomorphic terms in $z$
that do not arise from a constant multiple of $Q$.
\end{proof}

The exceptional unknown case with nondegenerate $A$ is the manifold
\begin{equation}
w = \sabs{z_1}^2 - \sabs{z_2}^2 + \lambda ( z_1^2 + z_2^2 + \bar{z}_1^2 +
\bar{z}_2^2 ) + O(\snorm{z}^3)
\end{equation}
for $\lambda \geq \frac{1}{2}$.  The theorem says that, to prove or disprove
flattenability,
one needs to find or show the nonexistence of a real-analytic function
\begin{equation}
g(z,\bar{z}) = \sabs{z_1}^2 - \sabs{z_2}^2 + \lambda ( z_1^2 + z_2^2 + \bar{z}_1^2 +
\bar{z}_2^2 ) + O(\snorm{z}^3)
\end{equation}
that is constant on the CR orbits of $M_{\mathit{CR}}$.

\section{CR singular images of CR manifolds} \label{sec:CRimages}

Let us motivate this section by 
Example 1.6 in Ebenfelt--Rothschild~\cite{ER}.
Let $M' \subset \C^3$ be given by
\begin{equation}
M' = \left\{(z, w_1 , w_2 ) \in \C^3 : \Im w_1 = \frac{{\abs{z}}^2}{2},
\Im w_2 = \frac{{\abs{z}}^4}{2} \right\} .
\end{equation}
The submanifold $M'$ is a generic (and hence CR) real-analytic submanifold of finite type (in particular, not Levi-flat).
The submanifold $M'$ is taken to the CR singular submanifold 
\begin{equation}
M = \{(z_1, z_2, w) \in \C^3 : w = {(\bar{z}_2+i{\abs{z_1}}^2+{\abs{z_1}}^4 )}^2 \}
\end{equation}
via the holomorphic map
\begin{equation}
(z, w_1, w_2) \mapsto
\bigl(z, w_1 + iw_2 , {(w_1 - iw_2 )}^2 \bigr) .
\end{equation}
The map restricted to $M'$ is a diffeomorphism onto its image.
A natural problem is to classify
such CR singular submanifolds $M$.

Let us define this class of manifolds more abstractly.
Let $M \subset \C^m$ be a real-analytic submanifold.  We say
$M$ is a \emph{CR image}
if there exists a real-analytic vector bundle
$\sV \subset \C \otimes TM$ such that $\sV_p = T_p^{0,1} M$ for
all $p \in M_{CR}$.
In other words, the CR structure of $M_{CR}$ extends to an
abstract CR structure on $M$.
As real-analytic CR structures
are always integrable (Theorem 2.1.11 in \cite{BER:book}), we have 
the following proposition.

\begin{prop}
Let $M \subset \C^m$ be a connected real-analytic submanifold that is generic at some point.  Then the
following are equivalent.
\begin{enumerate}[(i)]
\item $M$ is a CR image.
\item For each $p \in M$ there are a neighborhood $U$ of $p$,
a real-analytic generic submanifold $M' \subset \C^{m}$,
and a CR immersion $\varphi \colon M' \to \C^{m}$ such that
$\varphi(M') = M \cap U$.
\end{enumerate}
\end{prop}

In \cite{LMSSZ} it was proved that if a CR singular $M$ is a CR image, then
there exists a CR function that is not a restriction of a holomorphic
function.  Theorems~\ref{thm:mainext} and~\ref{thm:PolyExtn} then
give the following corollary.

\begin{cor} \label{cor:CRimages}
Let $(z,w) \in \C^n \times \C$ be the coordinates and,
near the origin, let $M \subset \C^{n+1}$ be a codimension-2 submanifold given by
\begin{equation}
w = \rho(z,\bar{z}) = 
z^* A z + \overline{z^t B z} + z^t C z +
E(z, \bar{z}),
\end{equation}
where $\rho$ is real-analytic,
$A,B,C$ are complex $n \times n$ matrices,  $B$ and $C$ are symmetric,
and $E$ is
$O(\snorm{z}^3)$.
If $M$ is a CR image, then the following hold.
\begin{enumerate}[(a)]
\item
$\rank \begin{bmatrix} A^* \\ B \end{bmatrix} \leq 1$.
\item
   Near the origin, $M$ is biholomorphically equivalent to exactly
   one of the following forms:
   \begin{enumerate}[(1)]
     \item $w = \bar{z}_1 z_2 + \bar{z}_1^2 + O(\snorm{z}^3)$,
     \item $w = \bar{z}_1 z_2 + O(\snorm{z}^3)$,
     \item $w = \sabs{z_1}^2 + a \bar{z}_1^2 + O(\snorm{z}^3)$, $a \geq 0$,
     \item $w = \bar{z}_1^2 + O(\snorm{z}^3)$,
     \item $w = O(\snorm{z}^3)$.
   \end{enumerate}
   Furthermore, examples exist for all five cases.
\end{enumerate}
\end{cor}

The Ebenfelt--Rothschild example above is equivalent to case (4) in the corollary.

The quadratic exceptional cases from Theorem~\ref{thm:PolyExtn} are 
such images of $\R^2 \times \C^{n-1}$; see Remark~\ref{remark:LFimages}, giving examples for cases (1)--(4).
For the fifth case,
consider $w = \bar{z}_1^3$ for $M$, $M' = \R^2 \times \C^{n-1}$, and
$(s,t,\xi) \mapsto \bigl(s+it,\xi,{(s-it)}^3\bigr)$ for the map.
All five cases are therefore possible.

The corollary says that any CR singular CR image
is either a third-order perturbation of
an image of the Levi-flat $\R^2 \times \C^{n-1}$ or a third-order perturbation of
an image of $\C^n$ ($w=0$).

%%%%%%%%%%%%%%%%%%%%%%%%%%%%%%%%%%%%%%%%%%%%%%%%%%%%%%%%%%%%%%%%%%%%%%%%%%%%

\def\MR#1{\relax\ifhmode\unskip\spacefactor3000 \space\fi%
  \href{http://www.ams.org/mathscinet-getitem?mr=#1}{MR#1}}

%%%%%%%%%%%%%%%%%%%%%%%%%%%%%%%%%%%%%%%%%%%%%%%%%%%%%%%%%%%%%%%%%%%%%%%%%%%%

\end{document}